\documentclass[11pt]{amsart}

\usepackage{geometry}                
\usepackage{graphicx,
	amssymb,
	amsmath,
	amsthm,
	bm,
	dsfont,
	epstopdf,
	mathrsfs,
	mathabx,
	multirow,
	mathtools,
	dsfont,
	xcolor,
	tabu,
	subfig,
	hyperref}
\hypersetup{
    colorlinks=true,
    allcolors=blue
}
\definecolor{redy}{rgb}{0.84, 0.48, 0.47}
\definecolor{greeny}{rgb}{0.45, 0.68, 0.49}
\definecolor{bluey}{rgb}{0.42, 0.6, 0.78}
\usepackage{amsopn, amsthm, amsgen, amscd,amsmath, amssymb}

\usepackage[all]{xy}

\newtheorem{theorem}[equation]{Theorem}
\newtheorem{proposition}[equation]{Proposition}
\newtheorem{corollary}[equation]{Corollary}
\newtheorem{lemma}[equation]{Lemma} \theoremstyle{definition}
\newtheorem{definition}[equation]{Definition}\theoremstyle{remark}
\newtheorem{remark}[equation]{Remark}
\numberwithin{equation}{section}

\title[Modular orbifold: Arithmetic and dynamics]{Arithmetic, geometry
and dynamics in the unit tangent bundle of the modular orbifold}
\author{Alberto Verjovsky}
\date{}

\newcommand{\Address}{{
\begin{center}
 \textsc{Instituto de Matem\'aticas, Unidad Cuernavaca\\
Universidad Nacional Aut\'onoma de M\'exico}\\
\end{center}
}}

\begin{document}
\maketitle

\Address

\vspace{12pt}

\noindent \textcolor{orange}{Updated version of the paper of the same title that appeared in:
\noindent Verjovsky, A.  \emph{Arithmetic geometry and dynamics in the unit tangent bundle of the modular orbifold.} Dynamical systems (Santiago, 1990), 263--298, Pitman Res. Notes Math. Ser., 285, Longman Sci. Tech., Harlow, 1993.}\\

\section*{Introduction}
The modular group $\text{PSL}(2,\mathbb{Z})$ and its action on the upper
half-plane $\mathbb{H}$ together with its quotient, the modular orbifold,
are fascinating mathematical objects. The study of modular forms has
been one of the classical and fruitful objects of study. If one
considers the group $\text{PSL}(2,\mathbb{R})$ and we take the quotient
$M=\text{PSL}(2,\mathbb{R})/\text{PSL}(2,\mathbb{Z})$ one obtains a three
dimensional manifold that carries an enormous amount of arithmetic
information. Of course this is not surprising as the elements of
$\text{PSL}(2,\mathbb{Z})$ are M\"obius transformations given by matrices
with columns consisting of lattice points with relatively prime
integer coefficients and therefore $M$ contains information about the
prime numbers. On the other hand from the dynamical systems viewpoint,
$M$ is also very interesting---it has a flow which is an Anosov
flow. The flow as well as its stable and unstable foliations
correspond to three one-parameter subgroups: the geodesic flow, and
the stable and unstable horocycle flows (see the next section). All
these flows preserve normalized Haar measure $\overline{m}$, and are
ergodic with respect to this measure.  By a theorem of Dani \cite{Da}
the horocycle flows have a curve $D(y),\;(y>0)$, of ergodic
probability measures. These ergodic measures are supported in closed
orbits of period $y^{-1}$ of the corresponding horocycle flows. If we
denote by $m_{y}$ the measure corresponding to $y>0$, then $m_{y}$
converges to $\overline{m}$ as $y\to 0$ (\cite{Da} \cite{Za}). We will
also give a proof of this. However what is most interesting for me as a person working on
a dynamical systems is the remarkable connection found by Don Zagier
between the rate of approach of $m_{y}$ to $\overline{m}$ an the
Riemann Hypothesis. Zagier found that the Riemann Hypothesis holds if
and only if for every smooth function $f$ whith compact support on $M$
one has $m_{y}(f)=\overline{m}(f)+o(y^{3/4-\epsilon})$ for all
$\epsilon>0$. We also proved that
$m_{y}(f)=\overline{m}(f)+o(y^{1/2})$. In the Appendix we will review
Zagier's approach as well as the extension given by Sarnak
\cite{Sa}. 

The purpose of this paper is to analyze the dynamics and
geometry of the horocycle flow to show that the exponent $1/2$ is
optimal for certain characteristic functions of sets called
``boxes''. Of course this is very far from disproving the Riemann
Hypothesis since a characteristic function is not even continuous. 

Our result puts in evidence the fact that the Riemann hypothesis is also a regularity problem.
At the end it is the lack of the Riemann-Lebesgue lemma for functions that are not in $L_1$. 
We remark that in Zagier paper \cite{Za} this point is not clarified.

By geometric means we reduce the analysis of the convergence of $m_{y}$
to $\overline{m}$ to a lattice point counting. Thus the fact that the
exponent cannot be made better than $1/2$ is similar to the circle
problem in which we count the number of lattice points with relatively
prime coordinates inside a circle. If instead we take a smooth
function $f$ with compact support and take the sum $\Sigma(y)$, of the
values of the function over all lattice points $(yn,ym)$ with $y>0$
and $(n,m)\in\mathbb{Z}^{2}$ then $y^{2}\Sigma(y)$ will converge to the
integral of $f$ over the plane as $y\to 0$ and the error term will be
$o(y^{\alpha})$, as $y\to 0$, for all $\alpha>0$. This follows by the
Poisson summation formula using the fact that the Fourier transform of
$f$ decays very rapidly at infinity.

We will use many properties of $ \text{SL}(2,\mathbb{R})$, its quotients by
discrete subgroups and the two locally free actions of the proper real
affine group on these quotients. Good references for each subject are
\cite{A}, \cite{G}, \cite{L}, \cite{Ma} and \cite{Ra}.

\section{Preliminares}
Let $\widetilde{G}:=\text{SL}(2,\mathbb{R})$ denote the Lie group of $2\times 2$
matrices of determinant one, with real coefficients. The Lie algebra
of $\widetilde{G}$, $\mathfrak{g}:=\mathfrak{sl(2,\mathbb{R})}$, consists of real $2\times 2$
matrices of trace zero. This Lie algebra has the standard basis:
\[A=\begin{bmatrix}
1/2 & 0 \\ 0 & -1/2
\end{bmatrix},\quad B=\begin{bmatrix}
0 & 1  \\
0 & 0
\end{bmatrix},\quad C=\begin{bmatrix}
0 & 0  \\
1 & 0
\end{bmatrix}.\]
To $A,\; B$ and $C$ correspond the left-invariant vector fields $X,\;
Y$, and $Z$ respectively in $\text{SL}(2,\mathbb{R})$. These vector fields
induce, respectively, the nonsingular flows:
\[g_{t}:\widetilde{G}\to\widetilde{G},\;\; h_{t}^{+}:\widetilde{G}\to\widetilde{G},\;\;h_{t}^{-}:\widetilde{G}
\to\widetilde{G},\;\; t\in \mathbb{R}.\]
Explicitly:
\begin{eqnarray}
\nonumber g_{t}\left(\begin{bmatrix} a & b \\ c & d
\end{bmatrix}\right)&=&\begin{bmatrix}
a & b \\ c & d
\end{bmatrix}\begin{bmatrix}
e^{\frac{t}{t}} & 0 \\ 0 & e^{-\frac{t}{2}}
\end{bmatrix},\\
\label{0.1}h_{t}^{+}\left(\begin{bmatrix}
a & b \\ c & d
\end{bmatrix}\right)&=&\begin{bmatrix}
a & b \\ c & d
\end{bmatrix}\begin{bmatrix}
1 & t \\ 0 & 1
\end{bmatrix},\\
\nonumber h_{t}^{-}\left(\begin{bmatrix} a & b \\ c & d
\end{bmatrix}\right)&=&\begin{bmatrix}
a & b  \\
c & d
\end{bmatrix}\begin{bmatrix}
1 & 0  \\
t & 0
\end{bmatrix},\quad t\in\mathbb{R}.
\end{eqnarray}
To simplify notation, let us write:
$g:=\{g_{t}\}_{t\in\mathbb{R}}$, $h^{+}:=\{h_{t}^{+}\}_{t\in\mathbb{R}},$ and
$h^{-}:=\{h_{t}^{-}\}_{t\in\mathbb{R}}$.

Consider the upper half-plane, $\mathbb{H}=\{z=(x,y):=x+iz\;|\;y>0\}
\subset\mathbb{C}$ equipped with the metric
$ds^{2}=(1/y^{2})(dx^{2}+dy^{2})$.  With this metric $\mathbb{H}$ is the
hyperbolic plane with constant negative curvature minus one.

$\widetilde{G}$ acts by isometries on $\mathbb{H}$ as follows:
\[ \begin{bmatrix}
a & b \\ c & d
\end{bmatrix}\begin{bmatrix}
x \\ y
\end{bmatrix}=\frac{az+b}{cz+d}\]
where $z=x+iy$, $y>0$.

The action is not effective and the kernel is the subgroup of order
two $\{I,-I\}$, consisting of the identity and its negative. 

Let
$G=\widetilde{G}/\{I,-I\}:=\text{PSL}(2,\mathbb{R})$. Note that $G$ is the group of
M\"obius transformations that preserve $\mathbb{H}$ and is in fact its full
group of orientation-preserving isometries.

The action of $G$ on $\mathbb{H}$ can be extended via the differential to
the unit tangent bundle which we shall denote henceforth by
$T_{1}\mathbb{H}$.  If $\gamma\in G$ and $\gamma'$ denotes its differential
acting on unit vectors, we have:
\[
\xymatrix{ T_{1}\mathbb{H} \ar[d]_{P_{1}} \ar[r]^{\gamma'} & T_{1}\mathbb{H}
  \ar[d]^{P_{1}}\\ \mathbb{H} \ar[r]_{\gamma} & \mathbb{H} }
\]
where $P_{1}$ is the canonical projection. Naturally, ``unit vector''
refers to the hyperbolic metric.

By euclidean translation to the origin and clockwise rotation of 90
degrees we have a trivialization $\psi:T_{1}\mathbb{H}\to\mathbb{H}\times S
^{1}$ given by
\begin{equation}\label{0.2}
 \psi(z,v)=(z,-iv/|v|),\quad (i^{2}=-1),
\end{equation}
where $v$ is a hyperbolic unit tangent vector anchored at $z\in
T_{1}\mathbb{H}$. For example, using (\ref{0.2}), $\gamma\mapsto
\gamma'(i,1)$ gives an explicit identification and it will be the one
we will use here.

If $\gamma=
\begin{bmatrix}
  a & b \\ c & d
\end{bmatrix}\in \widetilde{G}$,
we will let $\overline{\gamma}(z)= \frac{az+b}{cz+d}$ denote the
corresponding element in $G$. Using trivialization (\ref{0.2}), we
have
\begin{equation}\label{0.3}
 \gamma'(z,\theta)=(\overline{\gamma}(z),\theta-2\arg(cz+d)).
\end{equation}
In this notation $\theta$ is to be taken modulo $2\pi$, where the
angle of a unit vector is \textit{measured from the vertical
  counter-clockwise}.

The three basic vector fields $X$, $Y$ and $Z$ descend to $G\simeq
T_{1} \mathbb{H}$ and the flows induced by them correspond to the
\textit{geodesic flow, unstable horocycle flow and stable horocycle
  flow}, respectively.

Geometrically these flows can be described as follows. Let $z\in\mathbb{H}$
an let $v_{z}\in T_{1}\mathbb{H}$ be a unit vector based at $z$. This
vector determines a unique oriented geodesic $\gamma$, as well as two
oriented horocycles $C^{+}$ and $C^{-}$ which pass through $z$ are
orthogonal to $\gamma$ and tangent to the real axis. Then
$v':=g_{t}(v_{z})$ is the unit vector tangent to $\gamma$, following
the same orientation as $\gamma$ at the point at distance $t$ from
$z$. The vectors $w^{+}=h_{u}^{+}(v_{z})$ and $w^{-}=h_{v}(v_{z})$ are
obtained by taking unit vectors tangent to $C^{+}$ and $C^{-}$,
respectively, at distances $u$ and $v$ respectively, and according to
their orientations  (see Figure \ref{Figure1}).

\begin{figure}[h]
\begin{center}
  \includegraphics[width=12cm]{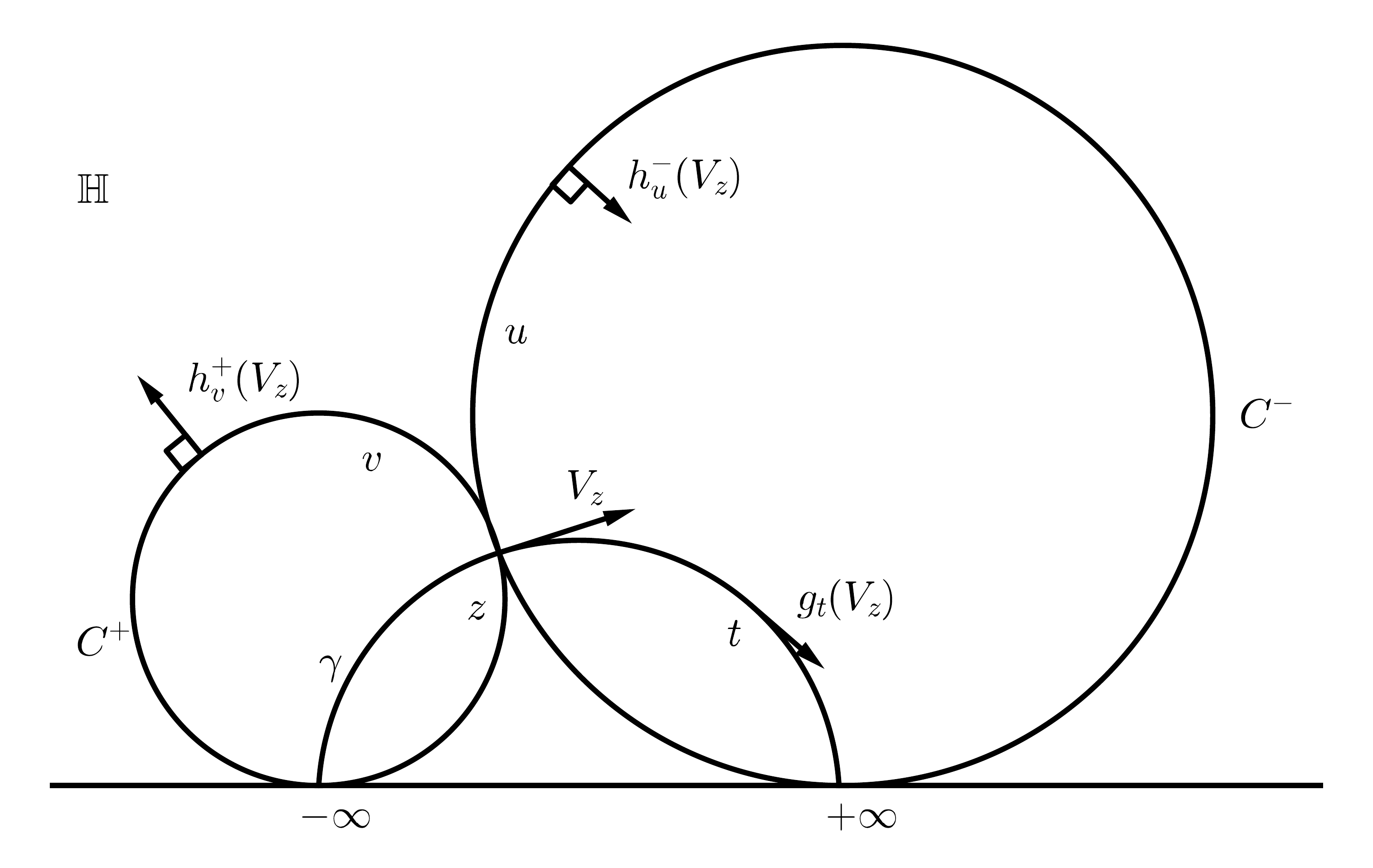}
\caption{}
\label{Figure1}
\end{center}
\end{figure} 

It is because of this geometric interpretation that the flows $g$,
$h^{+}$ and $h^{-}$, defined originally in $\widetilde{G}$ are called \emph{geodesic}
and \emph{horocycle} flows respectively. Formulae (\ref{0.1}) tell us that the
orbits of the  respective flows are obtained by left-translations of the
one-parameter subgroups
\[t\mapsto \begin{bmatrix}
e^{t/2} & 0 \\ 0 & e^{-t/2}
\end{bmatrix},\quad u\mapsto \begin{bmatrix}
1 & u  \\
0 & 1
\end{bmatrix},\quad v\mapsto \begin{bmatrix}
1 & 0  \\
v & 1
\end{bmatrix}.\]
Every noncompact one-parameter subgroup is conjugate in $\widetilde{G}$ to one
of the above, and $h^{+}$ is conjugate to $h^{-}$.

Henceforth we will equip $\widetilde{G}$ with the left-invariant Riemannian
metric such that $\{X,Y,Z\}$ is an oriented orthonormal framing. We
will call this metric the \textit{standard metric}.

By the \textit{standard Riemannian measure} or \textit{Haar Measure},
we will mean the measure, $m$, induced by the volume form $\Omega$ which
takes the constant value one in the oriented framing $\{X,Y,Z\}$. Since
$\widetilde{G}$ is unimodular, the measure $m$ is bi-invariant.

Let $A_{2}(\mathbb{R})$ denote the \textit{proper affine group}:
\[A_{2}(\mathbb{R})=\{T:\mathbb{R}\to\mathbb{R}\;|\;T(r)=ar+b;\;a,b\in\mathbb{R},\;a>0\}.\]

Let us parametrize $A_{2}(\mathbb{R})$ by pairs $(a,b)$ with $a,b\in\mathbb{R}$,
$a>0$.

There are two monomorphisms, $A_{2}(\mathbb{R})\hookrightarrow\widetilde{G}$, given
by:
\begin{equation}\label{0.4}
 (a,b)\mapsto \begin{bmatrix} a & b \\ 0 & a^{-1}
\end{bmatrix},\quad  (a,b)\mapsto \begin{bmatrix}
a & 0 \\ b & a^{-1}
\end{bmatrix}.
\end{equation}
We see from these inclusions that the pairs $\{X,Y\}$ and $\{X,Z\}$
generate Lie algebras isomorphic to the Lie algebra of the affine
group:
\[[X,Y]=Y,\quad[X,Z]=-Z,\quad[Y,Z]=X.\]
As a consequence, $\widetilde{G}$ contains two real analytic foliations by
planes whose leaves are the orbits of the free actions of the affine
group in $\widetilde{G}$ (or simply, the leaves are obtained by
left-translations of the two copies of the affine group). These two
foliations, denote respectively by $\mathcal{F}^{+}$ and $\mathcal{F}^{-}$,
intersect transversely along the orbits of the geodesic flow. The
geodesic flow is an Anosov flow an $\mathcal{F}^{+}$ and $\mathcal{F}^{-}$ are its
\textit{unstable} and \textit{stable} foliations.

Each leaf of $\mathcal{F}^{+}$ and $\mathcal{F}^{-}$ inherits from  the standard
metric, a metric of constant negative curvature equal to minus one. The
geodesic flow permutes the orbits of  $h^{+}$ as well as the orbits of
$h^{-}$. If  dilates uniformly and exponentially the orbits of $h^{+}$
and contracts uniformly and exponentially the orbits of $h^{-}$.

The flows $h^{+}$ and $h^{-}$ in the unit tangent bundle of $\mathbb{H}$ are
conjugate by  the so-called ``flip map'' which sends a unit tangent
vector to its  negative. \textit{therefore, any dynamical or ergodic
property that holds for $h^{+}$ also holds for $h^{-}$}.

For any discrete subgroup $\Gamma\subset G$, the basic vector fields
$X,\;Y$ and $Z$ descend to the quotient
\[M(\Gamma):=\text{PSL}(2,\mathbb{R})/\Gamma,\] and they induce flows which we
will still denote by $g,\;h^{+}$ and $h^{-}$.

The standard metric, invariant form $\Omega$ and Haar measure $m$
descend to $M(\Gamma)$. All three flows $g,\;h^{+}$ and $h^{-}$
defined in $M(\Gamma)$ preserve $m$. Therefore $g$ is a
volume-preserving Anosov flow.

When $\Gamma$ is a discrete, co-compact subgroup, then $g,\;h^{+}$ and
$h^{-}$ are all $m$-ergodic and $g$ is a topologically transitive
Anosov flow. In this case $M(\Gamma)$ is a Seifert bundle over a
compact two dimensional orbifold. The exceptional fibres are due to
the elliptic elements of $\Gamma$. Both foliations $\mathcal{F}^{+}$ and
$\mathcal{F}^{-}$ are transverse to the fibres and every leaf is dense.

When $\Gamma$ is co-compact both $h^{+}$ and $h^{-}$ are minimal flows
on $M(\Gamma)$. In particular, the horocycle flows do not contain
periodic orbits. This was proved by Hedlund \cite{He}. It is also a
consequence of a result of Plante \cite{Pl}. Suppose for instance that
one orbit of $h^{+}$ is not dense. Then the closure of this orbit
contains a non-trivial minimal set $\Sigma$ and Plante showed that
$\Sigma$ must be both a 2-torus and a global cross section for the
geodesic flow, implying that $M(\Gamma)$ would be a torus bundle over
$S^{1}$. But this is impossible since under the hypothesis, $\Gamma$
cannot be solvable.

It was shown by Furstenberg \cite{Fu} that for $\Gamma$ co-compact
both $h^{+}$ and $h^{-}$ are strictly ergodic flows with $m$ as their
unique invariant measure (this also implies minimality of $h^{+}$ and
$h^{-}$).  Since the geodesic flow is transitive, it contains a set of
the second Baire category of dense orbits and it also contains a
countable number of periodic orbits whose union is dense in
$M(\Gamma)$. The number of such periodic orbits as a function of their
periods grows exponentially.  Margulis has a formula relating the
topological entropy of $g$ and the rate of growth of periodic
orbits. This formula was also latter obtained by Bowen. These facts
are interesting because the horocycle flow is the limit of conjugates
of the geodesic flow as can easily be seen by considering the vector
fields $X_{\epsilon}=\epsilon X+Y$ as $\epsilon$ tends to zero.

When $\Gamma$ is a nonuniform lattice (i.e. when $\Gamma$ is discrete,
$M(\Gamma)$ is not compact, and $m(M(\Gamma))<\infty$) the $g,\;h^{+}$
and $h^{-}$ are still $m$-ergodic. However, for nonuniform lattices
both $h^{+}$ and $h^{-}$ are not minimal.

If $\Gamma\subset G$ is any discrete subgroup, then
$\Gamma\backslash\mathbb{H}:=S(\Gamma)$ is a complete hyperbolic orbifold.
If $\Gamma$ is co-compact this means that $S(\Gamma)$ is a compact
surface provided with a special metric and a finite number of
distinguished points or \textit{conical points} labelled by rational
integers. The complement of the conical points is isometric a surface
of curvature minus one (in general incomplete), and of finite area.
Each distinguished point has a neighborhood which is isometric to the
metric space obtained by identifying, isometrically, the two equal
sides of an isosceles hyperbolic triangle. These two equal sides have
to meet at an angle $2\pi p/q$, where $p/q$ is the rational number
attached to the distinguished point.  Of course, the distinguished
points correspond to the equivalence of fixes points of elliptic
elements of $\Gamma$.  $S(\Gamma)$ is obtained by identifying sides of
a fundamental domain which is a finite hyperbolic polygon in $\mathbb{H}$
by elements of $\Gamma$. 

A classical theorem of Selberg (which also
holds for $\text{PSL}(2,\mathbb{C})$) asserts that $\Gamma$ contains a
subgroup $\widetilde{\Gamma}$ of finite index and without elliptic
elements. Then $S(\widetilde{\Gamma})$ is a branched covering of $S(\Gamma)$
and it is a complete hyperbolic surface without conical points and
finite area. In the unit tangent bundle
$T_{1}(S(\widetilde{\Gamma}))=G/\widetilde{\Gamma}$ we have the geodesic and
horocycle flows and there is a finite covering map from
$T_{1}(S(\widetilde{\Gamma}))$ onto $G/\Gamma=M(\Gamma)$ whose deck
transformations commute with the three flows. In this way, for any
discrete $\Gamma$, we can now speak of the horocycle and geodesic
flows on an orbifold. $M(\Gamma)$ plays the role of the unit tangent
bundle of $S(\Gamma)$ when $\Gamma$ has elliptic elements.

If $\Gamma$ is a nonuniform lattice theh $S(\Gamma)$ is a noncompact,
complete, hyperbolic orbifold of finite area. The orbifold is obtained
by identifying isometrically, pairs of sides of an \textit{ideal}
hyperbolic polygon of finite area. It then follows that the
fundamental polygon has a finite number of sides and there must be at
least one vertex at infinity. Each vertex at infinity is the common
end point of two asymptotic geodesics which are identified by a
parabolic element of $\Gamma$ (see Figure \ref{Figure2}).

\begin{figure}[h]
\begin{center}
  \includegraphics[width=10cm]{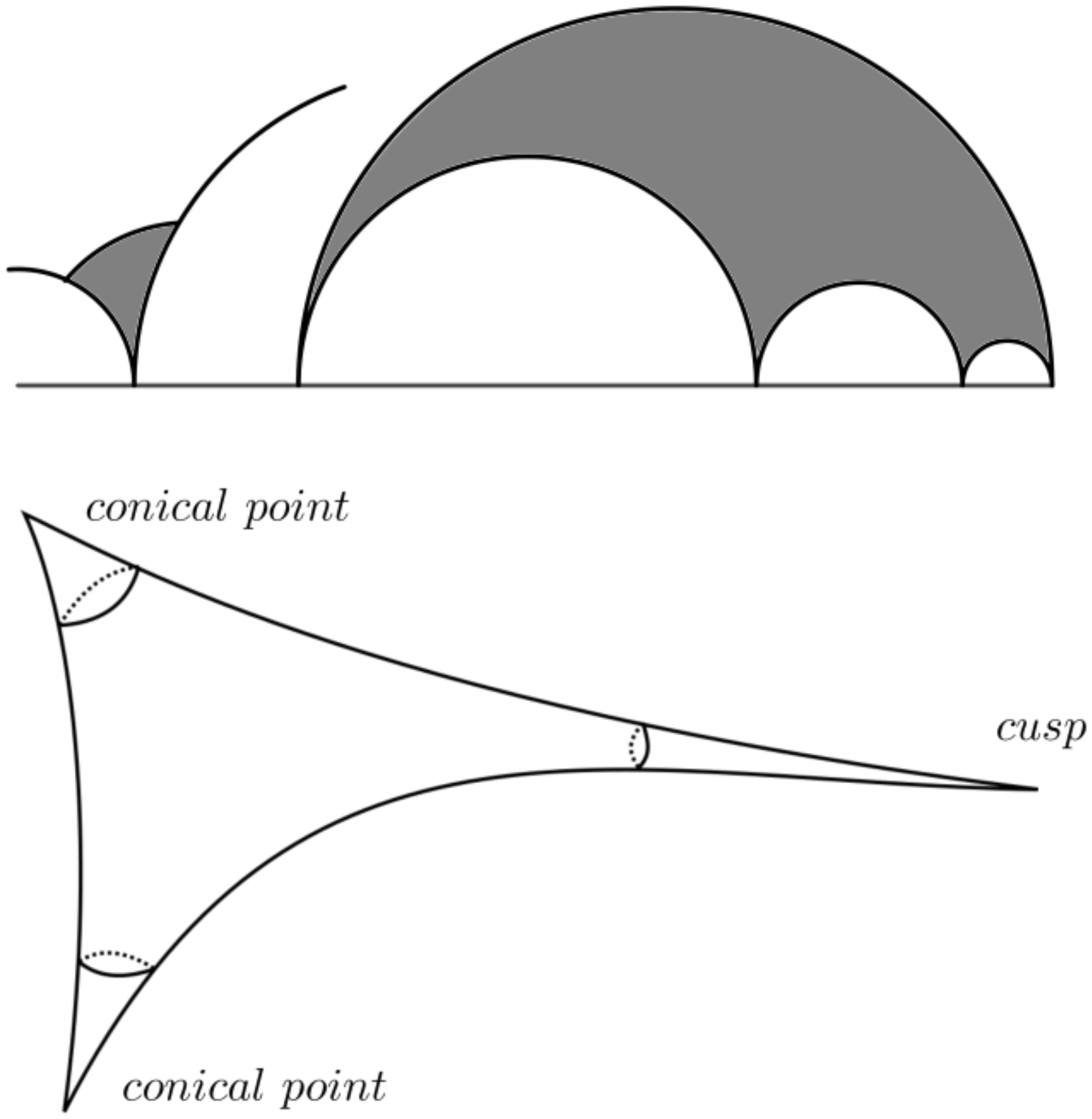}
\caption{}
\label{Figure2}
\end{center}
\end{figure}

In this way we obtain a ``cusp" for each equivalence class under
$\Gamma$ of the set of vertices at infinity. Also we obtain a
one-parameter family of closed horocycles in $S(\Gamma)$.

The orbifold $S(\Gamma)$ has a finite number of cusps and conical
points. If we compactify $S(\Gamma)$ by adding one point at infinity
for each cusp we obtain a compact surface $S(\Gamma)$. It is natural
to label these new points by the symbol $\infty$ (or by $\infty_{c}$
if we want to specify the cusp $c$) even though the ``angle'' at a
cusp is zero. 

The genus, area, number of cusps and conical points are
related by a Gauss-Bonnet type of formula.

In this paper the most important orbifold will be the \textit{ modular
  orbifold} which correspond to $\Gamma=\text{PSL}(2,\mathbb{Z})$. Its area
is, by Gauss-Bonnet formula, equal to $\pi/3$. An important number for
us will be $\pi^{2}/3$ which is the volume of $\text{PSL}(2,\mathbb{R})/
\text{PSL}(2,Z)$.\\

\noindent\textbf{Notation.} $M:=M(\text{SL}(2,\mathbb{Z}))=\text{PSL}(2,\mathbb{R})/
\text{PSL}(2,\mathbb{Z})$. \\

The modular orbifold is obtained from the standard modular fundamental
domain by identifying sides as shown in figure \ref{Figure3}.\\

\begin{figure}[h]
\begin{center}
  \includegraphics[width=12cm]{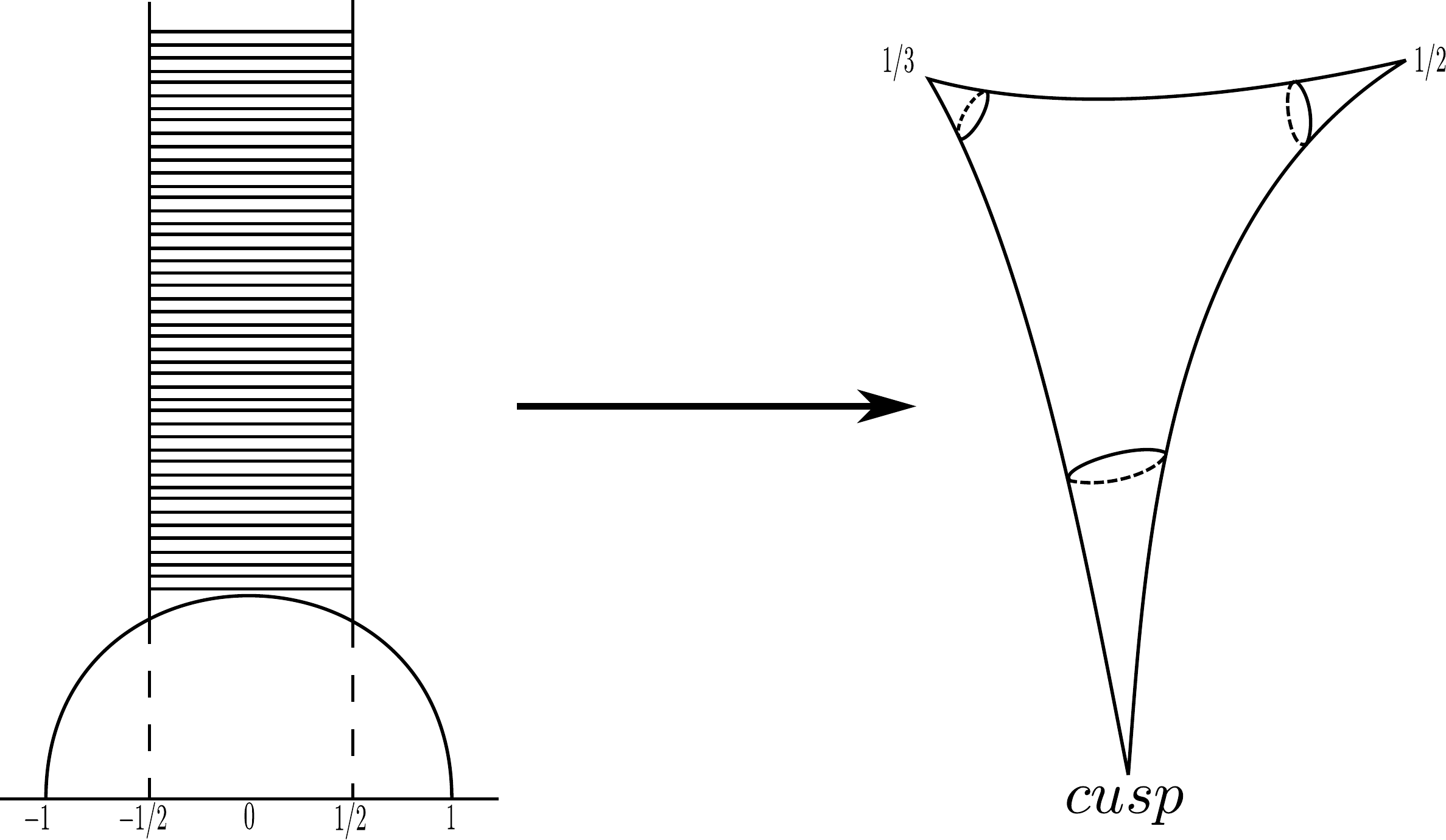}
\caption{}
\label{Figure3}
\end{center}
\end{figure}

The modular orbifold has just one cusp and two conical points with
labels $\tfrac{1}{2}$ and $\tfrac{1}{3}$. $M$ is a Seifert bundle over
the modular orbifold with two exceptional orbits corresponding to the
conical points. Both foliations $\mathcal{F}^{+}$ and $\mathcal{F}^{-}$ have dense
leaves. Let $\Gamma\subset G$ be any nonuniform lattice. Assume for
simplicity that $\Gamma$ does not have elliptic elements. Consider all
asymptotic geodesics corresponding to a given cusp. The union of all
these geodesic covers the orbifold $S(\Gamma)$. In fact, given any
points of $S(\Gamma)$ there exists a countable dense set of unit
tangent vectors at this point so that if we take a semi-geodesic
starting at the point and tangent to the one of these vectors, then it
will converge to the cusp. In the case of the modular group the angle
between any two of these vectors at any point where two such
semi-geodesics meet is a rational multiple of $2\pi$. \textit{ This is
  one of the reasons why the modular orbifold carries so much
  arithmetical information}. The horocycles corresponding to the cusp
are regularly immersed closed curves in $S(\Gamma)$. ``Near'' the cusp
all such horocycles are embedded. A necessary condition for such a
closed horocycle to be embedded is that its hyperbolic length is
sufficiently small. However when the hyperbolic lengths of these
closed horocycles are big, they are no longer embedded and they
self-intersect. 

\noindent The number of self-intersections grows as their
lengths grow and they tend to ``fill up'' all to $S(\Gamma)$. See figure \ref{Figure4}).\\

\begin{figure}[h]
\begin{center}
  \includegraphics[width=8cm]{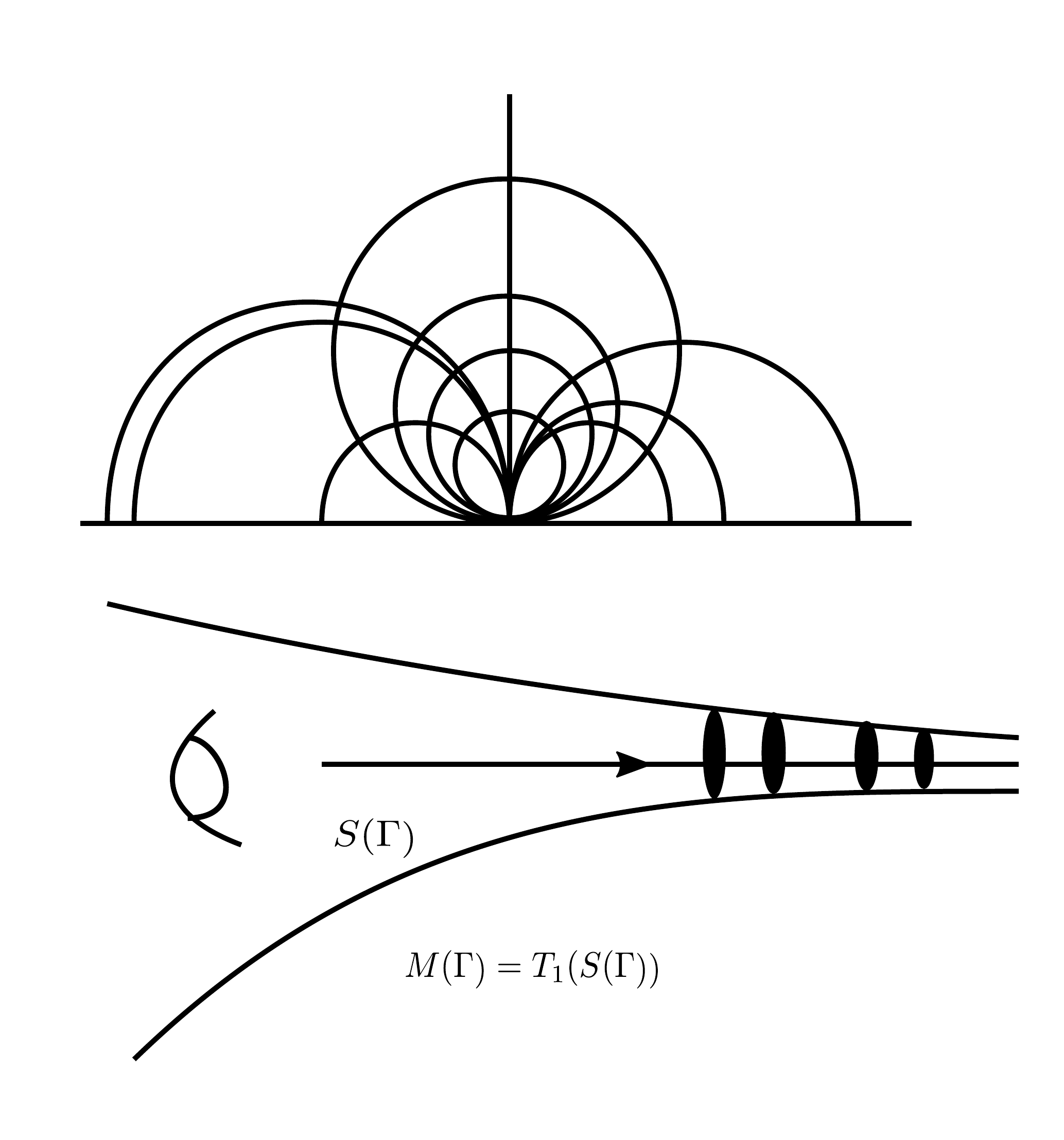}
\caption{}
\label{Figure4}
\end{center}
\end{figure} 

We can assume by conjugation in $G$, that the cusp is the point at
infinity in the upper half-plane so that the family of asymptotic
geodesics associated to the cusp is the family of parallel vertical
rays oriented in the upward direction. By conjugation we may also
assume that the parabolic subgroup associated to the cusp in the
following subgroup: \[\Gamma_{\infty}= \left.\left\{\pm\begin{bmatrix}
1 & n \\ 0 & 1 \end{bmatrix}\right|\ n\in\mathbb{Z}\right\}\] or
equivalently, the group of translations: $T_{n}(z)=z+n, \; n\in\mathbb{Z}$.
When the above happens we say that the point at infinity is the
\textit{standard cusp}. The modular group $\text{PSL}(2,\mathbb{Z})$ has the
standard cusp at infinity.

The horocycles corresponding to the standard cusp at infinity are the
horizontal lines in the upper half-plane. By formula (\ref{0.3}) we
see that the points $((x,y),\pi)$ and $((x+1,y),\pi)$ are identified
by the differential of elements of $\Gamma_{\infty}$. Therefore,
associated to the standard cusp there exists a one-parameter family of
periodic orbits of $h^{+}$ in the unit tangent bundle of $S(\Gamma)$:
\[\gamma_{y}=\{((x,y),\pi)\in T_{1}(S(\Gamma))\;|\;0\leq x\leq
1)\}\quad (y>0).\] The periodic orbit $\gamma_{y}$ has period (or
length) $1/y$.

In general since any cusp can be taken to the standard cusp at
infinity, we see that associated to the each cusp, there exists a
one-parameter family of closed orbits for the horocycle flow where
the natural parameter is the minimum positive period. Conversely,
given any periodic orbit $\gamma$ of $h^{+}$, then $\gamma$ is
included in such one-parameter family. This is so since the diameter
of $g_{-t}(\{\gamma\})$ tends to zero as $t$ tends to infinity and
thus it tends towards the point at infinity to some cusp. Exactly the
same reasoning plus the fact the $h^{+}$ is ergodic with respect to
$m$ implies that the only minimal subsets of $h^{+}$ are $M(\Gamma)$
and horocycle periodic orbits. Using the ``flip'' map we see that
everything we just said is also true for $h^{-}$.

As a consequence of the above remarks, we obtain that when $\Gamma$ is
a nonuniform lattice, there exist injective and densely immersed
cylinders $C_{i}$, $i=1,\dots,r$, where $r$ is the number of cusp.
These cylinders are mutually disjoint and their union comprises the
totality of periodic orbits of $h^{+}$. Naturally, these cylinders are
distinct leaves of the unstable foliation of the geodesic flow.
Everything that we have discussed so far follows by analyzing the
locally free actions of the affine group on $\text{SL}(2,\mathbb{R})$ and on
its quotients by discrete subgroups.

For nonuniform lattices $\Gamma$, the horocycle flows $h^{+}$ and
$h^{-}$ are not uniquely ergodic in $M(\Gamma)$ since both contain
periodic orbits. For each periodic orbit $\gamma$ of $h^{+}$ let
$m(\gamma)$ denote the Borel probability measure which is supported in
$\gamma$ and which is uniformly distributed with respect to its
arc-length i.e., if $f:M(\Gamma)\to\mathbb{R}$ is a continuous function, then
\[\langle m(\gamma),f\rangle:=m_{\gamma}(f)=
\frac{1}{T}\int_{0}^{T}f(h_{t}^{+}(x))dt,\] where $x\in\gamma$ is any
point in $\gamma$ and $T$ is the period of $\gamma$. It is evident that
$m(\gamma)$ is $h^{+}$-invariant and it is ergodic for $h^{+}$.

A theorem of Dani \cite{Da} asserts that the normalized standard
Riemannian measure and measures of type $m(\gamma)$ are the only
ergodic probability invariant measure for $h^{+}$. It also follows
from Dani's work that for a nonuniform lattice, an orbit of $h^{+}$
is either dense or else it is a periodic orbit.

Suppose that $\Gamma$ has only one cusp and that this cusp is the
standard cusp at infinity. Let $\widehat{M}(\Gamma)$ denote the
one-point compactification of $M(\Gamma)$. Let us extend $h^{+}$ by
keeping the point at infinity fixed. Let $m_{y}$ denote the ergodic
measure concentrated in the unique periodic orbit of period $1/y$. Let
$\delta_{\infty}$ denote the Dirac measure at the point at infinity.
Another result of Dani is that $m_{y}$ converges weakly to the
point-mass at infinity as $y\to\infty$ and $m_{y}$ converges weakly to
the normalized Haar measure $\overline{m}$, as $y\to 0$. The geometric
significance of this result is clear: as the period decreases, the
periodic orbit becomes smaller and tends to the cusp. On the other
hand, as the period increases, the horocycle orbit gets longer and
wraps around $M(\Gamma)$: it is almost dense. The latter case means
that as the period grows the horocycle orbits tend to be uniformly
distributed with respect to the normalized Haar measure.

Let $C^{0}(\widehat{M}(\Gamma))$ denote the Banach space of continuous
real-valued functions of $\widehat{M}(\Gamma)$, with the sup norm.
Let $C^{*}=[C^{0}(\widehat{M}(\Gamma))]^{*}$ denote its topological
dual with the $\text{weak}^{*}$-topology. Let
\[D:\mathbb{R}_{\bullet}^{+}\to C^{*};\quad D(y)=m_{y},\quad (y>0)\] where
$\mathbb{R}_{\bullet}^{+}$ denotes the multiplicative group of positive
real numbers.\\

\noindent\textbf{Dani's Theorem.} \cite{Da}. \textit{The measure
  $m_{y}$ converges in the weak$^{*}$-topology to the normalized Haar
  measure $\overline{m}$ as $y\to 0$. The only ergodic measures of
  $h^{+}$ are $D(y)$ for $y>0$, and $\overline{m}$}.\\

The $\text{weakly}^{*}$-compact convex envelope of the image of the
curve $D$ is the set of all invariant probability measures of
$h^{+}$.

The rate of approach of $m_{y}$ to the Haar measure $\overline{m}$
(as $y\to 0$) is intimately related to the Riemann Hypothesis. Don
Zagier \cite{Za} found a remarkable connection between the Riemann
hypothesis and the horocycles in the orbifold
$S(\text{PSL}(2,\mathbb{Z}))=\text{PSL}(2,\mathbb{Z})\backslash\mathbb{H}$ (i.e., in the
modular orbifold). Here we will use a particular case of Sarnak's
result \cite{Sa} (which generalizes the previously mentioned theorem
of Zagier) for the case $M=\text{PSL}(2,\mathbb{R})/\text{PSL}(2,\mathbb{Z})$:\\

\noindent\textbf{Theorem (Zagier).} \textit{Let $f$ any $C^{\infty}$
  function defined on $M$ and with compact support. Then:
\begin{equation} \label{0.5}
 m_{y}(f)=\overline{m}(f)+o(y^{1/2}),\quad (\text{as}\; y\to 0).
\end{equation}
Furthermore, the above error term can be made to be
$o(y^{3/4-\epsilon})$ for even $\epsilon>0$ if and only if the Riemann
Hypothesis is true}.\\

P. Sarnak \cite{Sa} developed Zagier's result to other discrete
subgroups of $\text{PSL}(2,\mathbb{R})$, namely nonuniform lattices. It will
be clear that many of the ideas contained in the present paper (which
incidentally, are reasonably simple and geometric) could be applied
to these more general cases, at least when the nonuniform lattices
are arithmetic (for example, for congruence subgroups).

In the present paper we will show that the exponent $1/2$ in the error
term of (\ref{0.5}) is optimal for certain characteristic functions
$\chi_{U}$. Namely, we will prove the following (see Theorem
\ref{Theorem 3.22}): \\

\noindent\textbf{Theorem.} \textit{There exists an open set $U\subset
M$ and a positive constant
$K>0$ depending only on $U$ such that:
\begin{equation}\label{0.6}
|m_{y}(\chi_{U})-\overline{m}(\chi_{U})|\leq Ky^{1/2}|\log y|,\quad
\text{for}\;\; 0<y\leq 1/2.
\end{equation}
Furthermore, if $\alpha>1/2$, then}\\
\noindent \text{(0.6')}\[\hspace{1.31 cm}\varlimsup_{y\to
0}[|m_{y}(\chi_{U})-\overline{m}(\chi_{U})|y^{-\alpha}]=\infty.\]\\

\noindent\textbf{Corollary.} \textit{The measure $m_{y}$ becomes
  uniformly distributed with respect to Haar measure as $y\to 0$:
\[\lim_{y\to 0}m_{y}(f)=\overline{m}(f)\quad \forall f\in
C^{0}(\widehat{M}(\Gamma)).\]
(See also \cite{Da} and \cite{Za}.)}\\

\noindent\textit{Remark.} Equality $(0.6')$ does not imply that the
Riemann Hypothesis is false since a characteristic function is not
even continuous.

\section{Geodesic and horocycle flows}
Let $U\subset M$ be a nonempty open set. Let $\gamma_{y}$ be the
horocycle orbit of period $1/y$. Then $\gamma_{y}\cap U$ is a
(possibly infinite) union of open intervals in $\gamma_{y}$ and
$m_{y}(U)$ is the sum of the lengths of these intervals divided by the
length of $\gamma_{y}$. For arbitrary $U$, to estimate how this grows
as $y\to 0$ seems an impossible task. However, if we take special open
sets this is indeed possible.

By a \textit{standard box} (or simply a \textit{box}) we will mean an
open set in $M$ which consists of the union of segments of orbits of
$h^{+}$ of equal length and whose middle point is contained in a open
``square'' $B$ (called the \textit{base} of the box). The base $B$ is
contained in a stable leaf $L^{-}\in\mathcal{F}^{-}$, of the stable
foliation of the geodesic flow. The boundary of $B$ consists of two
segments of geodesic orbits and two segments of stable horocycle
orbits. See \ref{Figure5}).\\

\begin{figure}[h]
\begin{center}
  \includegraphics[width=8cm]{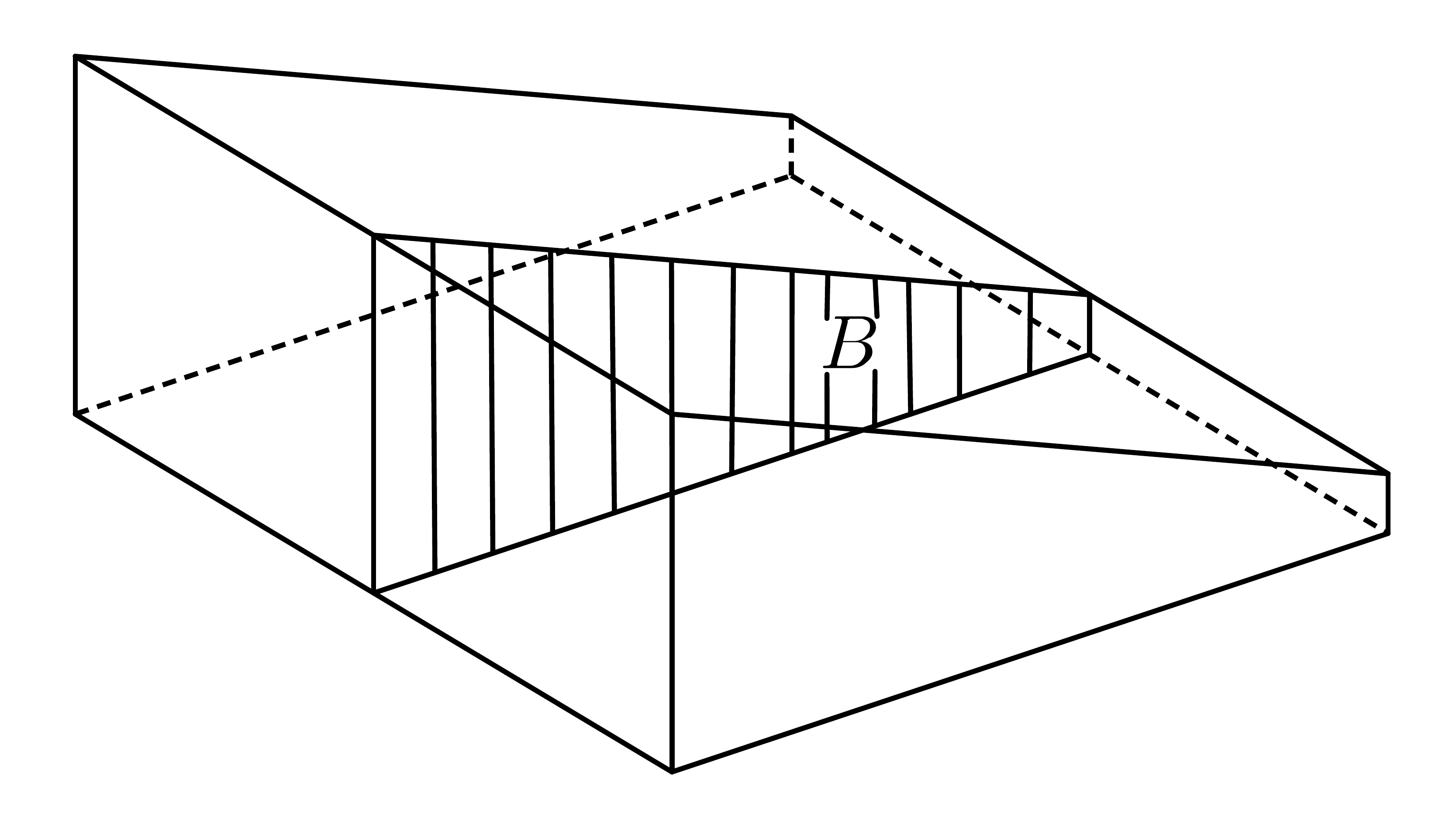}
\caption{}
\label{Figure5}
\end{center}
\end{figure}

The base and the common length can be chosen small enough so as to
have an embedded cube. In order to compute the volume (with respect to
the normalized Haar measure) of a standard box it is better to take a
lift of the box in the covering $\text{SL}(2,\mathbb{R})$, compute the volume
there, and then normalize. Let $C\subset\text{SL}(2,\mathbb{R})$ be a box,
let $\ell$ be the common length of the unstable segments of the box
(to be called the ``height'' of the box). Denote by $A$ the hyperbolic
area of the base $B$. Then, as we will show
latter: \[m(C)=A\ell=(\text{area of base})\times\text{``height''}.\] A
formula reminiscent of our elementary school days!

Therefore, if $C$ is a box in $M$ we obtain by normalization:
\[\overline{m}(C)=\frac{3}{\pi^{2}}(A\ell).\] The set of all boxes in
$M$ is a basis of the topology of $M$ and generates the
$\sigma$-algebra of its Borel sets. Also, the image under the geodesic
flow of a standard box is another standard box of equal volume (notice
that a standard box is a foliated chart for the unstable foliation but
not for the stable foliation).

Let $\gamma_{0}$ be the \textit{basis unstable horocycle periodic
orbit}, namely, the unique closed orbit of \[h_{t}^{+}:M\to M\] which
has period one. Then $\gamma_{t}:=g_{t}(\gamma_{0})$, as $t$ runs over
the real numbers, is the set of all closed orbits of $h^{+}$. Cleary,
$\gamma_{t}$ has length and period $e^{t}$. Therefore, we have the
change of parameter $y=e^{-t}$.

As $t$ grows, $\gamma_{t}$ becomes longer and starts ``filling up''
$M$, and it tends to be uniformly distributed with respect to the Haar
measure. Let $C\subset M$ be a box with base $B$ and let $A$ and
$\ell$ denote the area of the base and the height of $C$ respectively.
By the ergodic theorem no matter how small $C$ is, at some time $T>0$,
$\gamma_{t}$ will intersect $C$ for all $t>T$.

We have the following formula:
\begin{equation}\label{1.1}
 m_{y}(C)=n(y,C)y\ell;\quad y=e^{-t}.
\end{equation}
For all real $t$, $\gamma_{t}$ intersects $B$ transversally in a
finite number of points $n(y,C)$--the number that appears in the above
formula. This formula is evident: $\gamma_{t}\cap C$ is a finite
disjoint union if intervals of equal length $\ell$. The number of
these intervals is precisely $n(y,C)$ and we must divide by the length
of $\gamma_{t}$ which is $y^{-1}=e^{t}$ (see Figure \ref{Figure6}). \\

\begin{figure}[h]
\begin{center}
  \includegraphics[width=8cm]{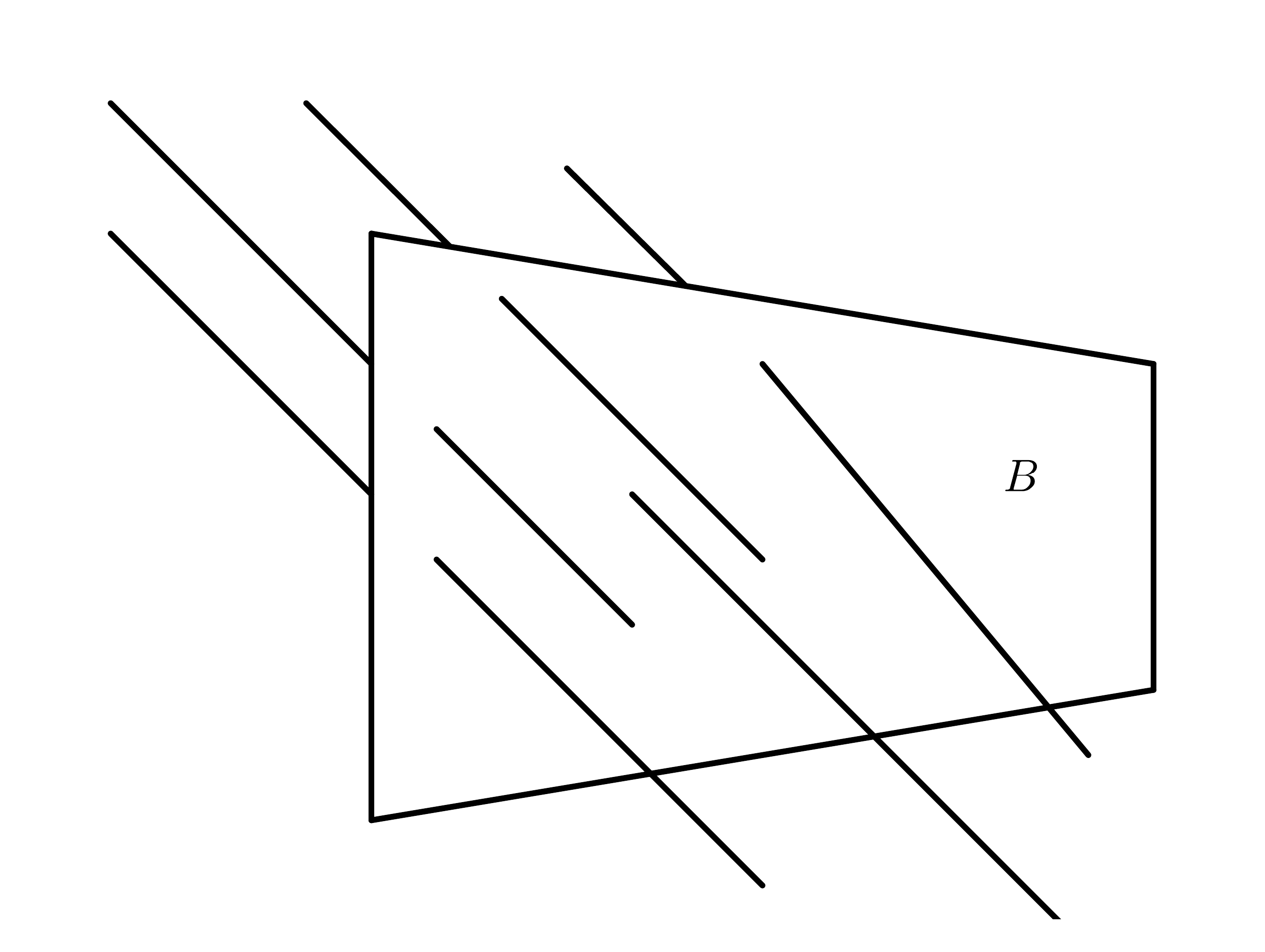}
\caption{}
\label{Figure6}
\end{center}
\end{figure}

Figure 6 suggests Amp\`ere's Law of electromagnetism. Let us imagine a
steady unit current flowing around the horocycle orbit $\gamma_{t}$
then the normalized integral of the magnetic field induced in the
boundary of the square is $n(y,C)$.

There is another way in which we can compute $m_{y}(C)$. Let
$C_{t}=g_{-t}(C)$, $(t>0)$, be the image of a standard box by the
geodesic flow reversing time. Then, since the geodesic flow preserves
$\overline{m}$, we have:
\begin{equation}\nonumber
\begin{cases}
g_{t}(C_{t}\cap\gamma_{1})=C\cap\gamma_{t}\\ \overline{m}(C_{t})=\overline{m}(C);\quad
t\in\mathbb{R}.
\end{cases}
\end{equation}

Let $B(t)$ be the base of $C_{t}$ and let $A(t)$ be its area. Let
$\ell(t)$ be the height of $C_{t}$. Then:
\[A(t)=e^{t}A(0)=e^{t}A,\]
where $A$ is the area of the base of $C$. We also have:
\[\ell(t)=e^{-t}\ell(0)=e^{-t}\ell,\] where $\ell$ is the height of $C$.

So as $t$ goes to $\infty$, the box $C_{t}$ becomes very thin and the
area of its base grows exponentially with $t$. There exists $T>0$ such
that $C_{t}$ intersects the basic horocycle orbit $\gamma_{0}$ for
all $t>T$. When this happens, $C_{t}\cap\gamma_{0}$ is a finite union
of intervals in $\gamma_{0}$ of equal length $\ell e^{-t}$. The number
of such intervals is $n(y,C)$ (see Figure \ref{Figure7}).\\

\begin{figure}[h]
\begin{center}
  \includegraphics[width=8cm]{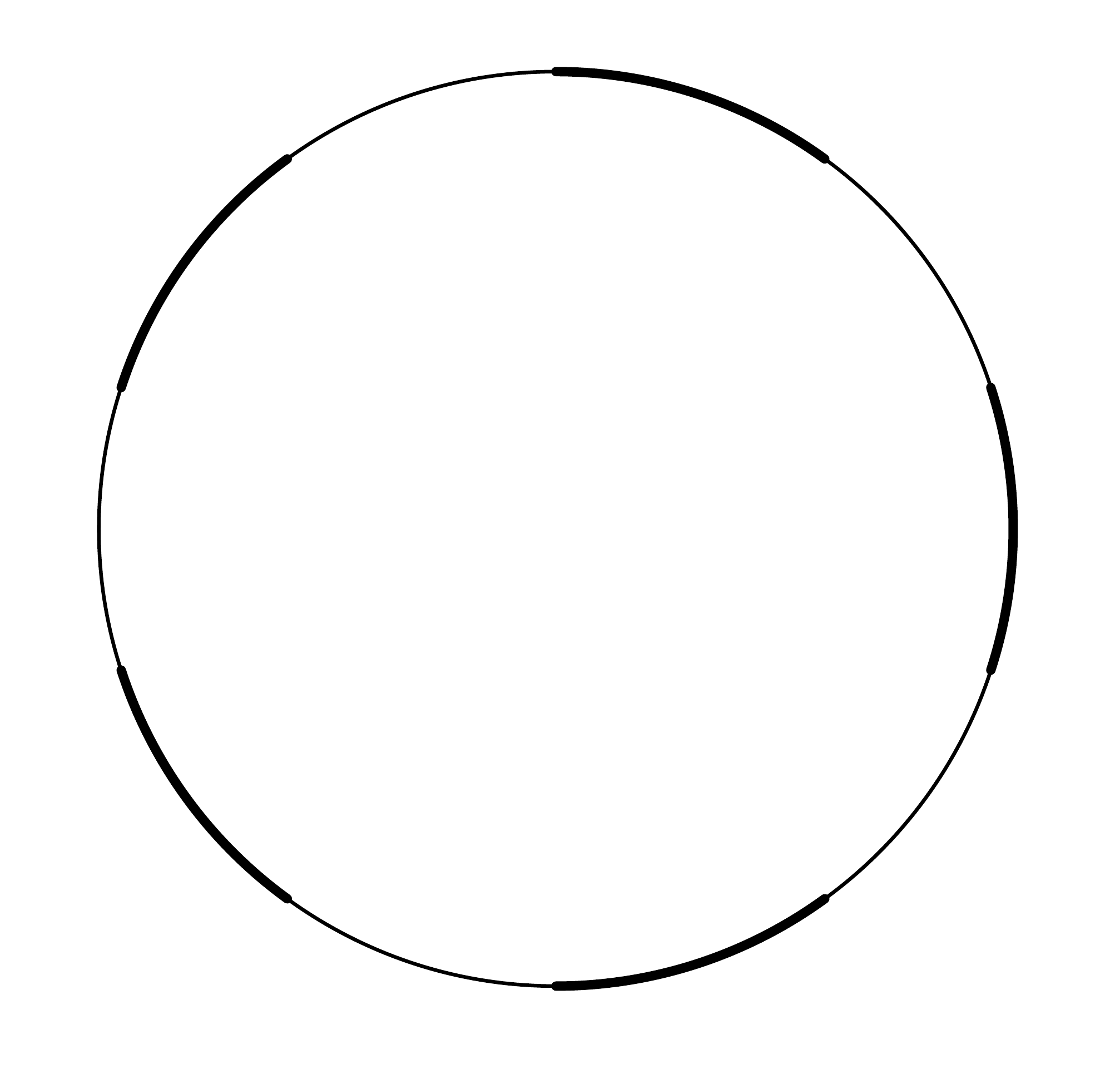}
\caption{}
\label{Figure7}
\end{center}
\end{figure}

Let $J(y)=\{J_{1}(y),J_{2}(y),\dots,J_{n(y,C)}(y)\}$ be this
collection of open intervals of equal length. The distribution of
$J(y)$ in the basic horocycle orbit looks seemingly random. This is
not exactly the case: the midpoints of the intervals are like a
circular Farey sequence of order $y^{-1}$. The connection between the
growth and distribution of circular Farey sequences and the Riemann
Hypothesis are well-known theorems of Franel \cite{Fr} and Landau
\cite{Lan} (see also Edwards \cite[p. 263]{E}). The problem of
understanding the distribution of $J(y)$ can be viewed as a problem of
equidistribution in the sense of Hermann Weyl.

We shall prove later that $J(y)$ has the pattern just described. In
fact, we will show:
\begin{equation}\label{1.2}
 n(y,C)=\frac{3}{\pi^{2}}y^{-1}A+O(y^{1/2}\log y)\quad (\text{as}\;
 y\to 0).
\end{equation}
Therefore, we see from (\ref{1.2}) that:
\[\lim_{y\to 0}[yn(y,C)]=\frac{3}{\pi^{2}}A.\]
Using (\ref{1.1}) we obtain:
\begin{equation}\label{1.3}
 \lim_{y\to 0}m_{y}(C)=\frac{3}{\pi^{2}}A\ell=\overline{m}(C).
\end{equation}
Formula (\ref{1.3}) implies that $m_{y}$ converges vaguely to
$\overline{m}$ as $y\to 0$.

However, what is important for us is that the exponent $1/2$ in the
error term in (\ref{1.2}) is optimal. This will follow from the fact
that $n(y,C)$ grows as a classical arithmetic function: the summatory
of Euler's $\varphi$ function:
\[\Phi(N)=\sum_{i=1}^{N}\varphi(i);\quad \Phi(r)=
\sum_{n\leq r}\varphi(n)\quad\text{if}\;\; r\in\mathbb{R},\; r\geq 1.\]
A classical theorem of Mertens (1874) shows that:
\[\Phi(r)=\frac{3}{\pi^{2}}r^{2}+0(r\log r),\quad\text{as}\;
r\to\infty.\] We will apply Merten's theorem to prove:

\begin{theorem}\label{T1.4} (Compare with Theorem \ref{Theorem 3.18})
  With the above notation we have,
  \[n(y,C)\sim\Phi(y^{-1/2}),\]
  where $\sim$ denotes asymptotic equivalence as $y\to 0$.
\end{theorem}
Therefore, Merten's theorem will show the validity of (\ref{1.2}).
Incidentally, the number
$\tfrac{3}{\pi^{2}}=\tfrac{1}{2}[\zeta(2)]^{-1}$ which appears in
Mertens'  formula is equal to the volume of $M$. \textit{Hence, weak
convergence $m_{y}\to\overline{m}$ (as $y\to 0$) is equivalent to
Mertens' theorem}. This theorem appears in any standard text book in
number theory, for instance, Apostol \cite{Ap}, Hardy and Wright
\cite{HW}  and Chandrasekharan \cite{Ch}. See also \cite{Da}.

\section{Some lemmas}
\begin{definition}
  Let $\Omega$ be the volume form in $\widetilde{G}$ such that:
\begin{equation*}
  \Omega_p(X_p, Y_p,Zp) = 1, \forall p \in \widetilde{G};
\end{equation*}
\noindent where $X_p$, $Y_p$, $Z_p \in T_p\widetilde{G}$ are tangent
vectors at $p$. Let $\langle \cdot, \cdot \rangle$ be the Riemannian
metric such that $\{X, Y, Z \}$ is an oriented orthonormal
framing. Let $\| \cdot \|$ be the norm induced by this metric.
\end{definition}

Throughout this paper we will use this metric. The measure $m$
determined by $\Omega$ is the \emph{standard Riemannian measure or
  Haar measure}.  \par If $\Gamma \subset \text{SL}(2,\mathbb{R})$ is any
discrete subgroup then $X$, $Y$, $Z$, $\Omega$, $m$ and $\langle
\cdot, \cdot \rangle$ descend to $\widetilde{M}(\Gamma)\coloneqq
\text{SL}(2,\mathbb{R})/\Gamma$, and they also descend to any quotient
$\text{PSL}(2, \mathbb{R})/\Gamma$.

\begin{definition}
Let $M = \text{PSL}(2, \mathbb{R})/\text{PSL}(2, \mathbb{Z})$. We will let
$\overline{m}$ denote the normalized Haar measure: $\overline{m} =
(3/\pi^2)m$.
\end{definition}
Using the monomorphisms of the affine group into $\widetilde{G}$ given
by formulae (\ref{0.4}), we obtain the commutation rule:
\begin{equation}\label{equation1}
\begin{bmatrix}
a & 0 \\
0 & a^{-1}
 \end{bmatrix}
\begin{bmatrix}
1 & b \\
0 & 1
\end{bmatrix}
=
\begin{bmatrix}
1 & a^{2}b \\
0 & 1
\end{bmatrix}
\begin{bmatrix}
a & 0 \\
0 & a^{-1}
\end{bmatrix}.
\end{equation}
The modular function of the affine group is:
\begin{equation*}
\alpha(a,b) = a^2.
\end{equation*}
\par As a manifold $\widetilde{G}$ is the product $S^1 \times
\mathbb{R}^2$. The group $\widetilde{G}$ can be decomposed as a product in
two ways using Iwasawa's decompositions: $\widetilde{G} =
\mathcal{N}\mathcal{A}\mathcal{K}$ and $\tilde{G} = \mathcal{A}\mathcal{N}\mathcal{K}$ where $\mathcal{N}$ is
the nilpotent group of matrices:
\begin{equation*}
\mathcal{N}= \Bigg\{
\begin{bmatrix}
1 & x\\ 0 & 1
\end{bmatrix} \Bigg | x \in \mathbb{R} \Bigg\},
\end{equation*}

$\mathcal{A}$ is the diagonal group:
\begin{equation*}
\mathcal{A}= \Bigg\{
\begin{bmatrix}
a & 0\\ 0 & a^{-1}
\end{bmatrix} \Bigg | a > 0 \Bigg\},
\end{equation*}

and $\mathcal{K}$ is the compact circle group:
\begin{equation*}
\mathcal{K} =
 \Bigg\{
r(\theta) \coloneqq
\begin{bmatrix}
\cos\theta & \sin\theta\\
-\sin\theta & \cos\theta
\end{bmatrix} \Bigg | \theta \in \mathbb{R} \Bigg\}.
\end{equation*}

Therefore any
$\begin{bmatrix}
a & b\\
c & d
\end{bmatrix} $
$\in \widetilde{G}$ can be written in a unique way as follows:
\begin{equation*}
\begin{bmatrix}
a & b\\
c & d
\end{bmatrix}
=
\begin{bmatrix}
1 & x\\
0 & 1
\end{bmatrix}
\begin{bmatrix}
y & 0\\
0 & y^{-1}
\end{bmatrix}
\begin{bmatrix}
\cos\theta & \sin\theta\\
-\sin\theta & \cos\theta
\end{bmatrix}
\end{equation*}
where $x, \theta \in \mathbb{R} $ and $y>0$.
\par Also it can be written in a unique way (using(\ref{equation1})) as
follows:
\begin{equation*}
\begin{bmatrix}
a & b\\
c & d
\end{bmatrix}
=
\begin{bmatrix}
y & 0\\
0 & y^{-1}
\end{bmatrix}
\begin{bmatrix}
1 & y^{-2}x\\
0 & 1
\end{bmatrix}
\begin{bmatrix}
\cos\theta & \sin\theta\\
-\sin\theta & \cos\theta
\end{bmatrix}.
\end{equation*}
\par These two parametrizations of $\widetilde{G}$ give different
expressions for the Haar measure written in terms of $dx$, $dy$ and
$d\theta$.
\par If $\begin{bmatrix}
a & b\\
c & d
\end{bmatrix}$ $\in \widetilde{G}$, then:
\begin{equation*}
\begin{bmatrix}
a & b\\
c & d
\end{bmatrix}
=
\begin{bmatrix}
(c^{2} + d^{2})^{-1/2} & 0\\
0 & (c^2 + d^2)^{1/2}
\end{bmatrix}
\begin{bmatrix}
1 & u\\
0 & 1
\end{bmatrix}
\begin{bmatrix}
\cos\theta & \sin\theta\\
-\sin\theta & \cos\theta
\end{bmatrix}
\end{equation*}
where $\cos\theta = d(c^2 + d^2)^{-\frac{1}{2}}$, $\sin\theta= c(c^2 +
d^2)^{-\frac{1}{2}}$ and $u = \frac{1}{d}[b(c^2 + d^2) + c]$ if $d
\neq 0$. If $d = 0$ then $bc = -1$, and if $c>0$ we have
\begin{equation*}
\begin{bmatrix}
a & b\\
c & 0
\end{bmatrix}
=
\begin{bmatrix}
c^{-1} & 0\\
0 & c
\end{bmatrix}
\begin{bmatrix}
1 & ca\\
0 & 1
\end{bmatrix}
\begin{bmatrix}
0 & -1\\
1 & 0
\end{bmatrix}.
\end{equation*}
This give the explicit $\mathcal{ANK}$ decomposition.
\par We also have, if $d\neq 0$,
\begin{equation*}
\begin{bmatrix}
a & b\\
c & d
\end{bmatrix}
=
\begin{bmatrix}
1 & \frac{1}{d}[b+c(c^2 + d^2)^{-1}]\\
0 & 1
\end{bmatrix}
\begin{bmatrix}
(c^2 + d^2)^{-\frac{1}{2}} & 0\\
0 & (c^2 + d^2)^{\frac{1}{2}}
\end{bmatrix}
\begin{bmatrix}
\cos\theta & \sin\theta\\
-\sin\theta & \cos\theta
\end{bmatrix}
\end{equation*}
if $d= 0$,
\begin{equation*}
\begin{bmatrix}
a & b\\
c & 0
\end{bmatrix}
=
\begin{bmatrix}
1 & ac^{-1}\\
0 & 1
\end{bmatrix}
\begin{bmatrix}
c^{-1} & 0\\
0 & c
\end{bmatrix}
\begin{bmatrix}
0 & -1\\
1 & 0
\end{bmatrix}.
\end{equation*}
Let $A = \begin{bmatrix}
a & b\\
c & d
\end{bmatrix}$ $\in \widetilde{G}$, and let us write the unique
decomposition
\begin{equation*}
A
=
\begin{bmatrix}
1 & x\\
0 & 1
\end{bmatrix}
\begin{bmatrix}
y & 0\\
0 & y^{-1}
\end{bmatrix}
\begin{bmatrix}
\cos\theta & \sin\theta\\
-\sin\theta & \cos\theta
\end{bmatrix};
\quad y>0.
\end{equation*}
\par Then the map $\psi \colon \text{SL}(2, \mathbb{R}) \to T_1\mathbb{H}$, which
identifies $\text{SL}(2,\mathbb{R})$ as a double covering of $\text{PSL}(2,\mathbb{R})
\coloneqq T_1\mathbb{H}$ is given by
\begin{equation*}
\psi(A) = (x + iy^2, \frac{\pi}{2} - 2\text{arg}(ci + d)).
\end{equation*}
\par Therefore the measure induced in $T_1\mathbb{H}$ by this identification
is given by the volume form
\begin{equation*}
d\text{V} = \frac{1}{2} \frac{dx\; dy\; d\theta}{y^2}.
\end{equation*}
Hence, if $U \subset T_1\mathbb{H}$ is any open set then by Fubini's theorem
we have:
\begin{equation}\label{equation2}
m(U) = \int_{U}d\text{V} = \frac{1}{2}\int_{0}^{2\pi}\text{A}_h(\text{U}
(\theta))d\theta,
\end{equation}
where $\text{U}(\theta)= \{z \in \mathbb{H}|\quad (z,\theta)\in \text{U} \}$ is
the ``slice" of $U$ corresponding to $\theta$ and
\begin{equation*}
\text{A}_h(U(\theta))= \iint_{U(\theta)}\frac{dx \; dy}{y^2}
\end{equation*}
is its hyperbolic area.
\par From Formula (\ref{equation2}) we obtain that if $U \subset
T_1\mathbb{H}$ is $S^{1}$-saturated (i.e. it is a union of circle fibers),
$U= D\times S^{1}$, then
\begin{equation*}
m(U) = \pi\text{A}_h(D).
\end{equation*}
In particular if we take the fundamental domain of the action of
$\text{PSL}(2,\mathbb{Z})$ in $T_1\mathbb{H}$ consisting of all unit tangent vectors
based in the modular fundamental domain in $\mathbb{H}$ we have:
\begin{equation*}
m(M)= 3/\pi^2.
\end{equation*}
Now let $\text{A}=\text{A}(u,t,v)\in \widetilde{G}$ be as follows:

\begin{equation}\label{equation3}
\text{A}
=
\begin{bmatrix}
1 & u\\
0 & 1
\end{bmatrix}
\begin{bmatrix}
e^{t/2} & 0\\
0 & e^{-t/2}
\end{bmatrix}
\begin{bmatrix}
1 & 0\\
v & 1
\end{bmatrix}
=
\begin{bmatrix}
e^{t/2} + uve^{-t/2} & ue^{-t/2}\\
ve^{-t/2} & e^{-t/2}
\end{bmatrix}.
\end{equation}
then $\alpha(\text{A})= e^{t/2}(v^2 + 1)^{-1/2}$, so we have:
\begin{eqnarray}\label{equation4}
\text{A} & = &
\begin{bmatrix}
1 & u + e^{t}(v^2 + 1)^{-1}v\\
0 & 1
\end{bmatrix}
\begin{bmatrix}
e^{t/2}(v^2 + 1)^{-1/2} & 0\\
0 & e^{-t/2}(v^2 + 1)^{1/2}
\end{bmatrix}\\
& &\cdot
\begin{bmatrix}
(v^2 + 1)^{-1/2} & -v(v^2 + 1)^{-1/2} \\
v(v^2 + 1)^{-1/2} & (v^2 + 1)^{-1/2}
\end{bmatrix}\nonumber
\end{eqnarray}
Then (\ref{equation4}) is the $\mathcal{NAK}$ decomposition of $A$ and we have:
\begin{equation*}
\cos\theta = (v^2 + 1)^{-1/2}, \quad \sin\theta = -v(v^2 + 1)^{-1/2},
\quad \text{ and } v= -\tan\theta.
\end{equation*}
Let $\widetilde{U} \subset \text{SL}(2, \mathbb{R})$ be the closed set:
\begin{equation*}
\widetilde{U} = \{\text{A}(u,t,v)|u_0 \leq u \leq u_1, v_0 \leq v \leq v_1,
t_0\leq t \leq t_1 \}
\end{equation*}
where $\text{A}(u,t,v)$ is a defined in (\ref{equation3}).
\par Let $U \subset \text{PSL}(2,\mathbb{R})$ be the projection of $U$. Then:
\begin{equation*}
m(U)= \frac{1}{2}\int_{\theta_1}^{\theta_0} k(\theta)d\theta
\end{equation*}
where $\theta_0$, $\theta_1 \in (-\pi/2, \pi/2)$ are the unique numbers
such that $v_0 = -\tan\theta_0$ and $v_1 = -\tan\theta_1$, and where
$k(\theta) = \text{A}_h(U(\theta))$.
\par Using formula (\ref{equation4}) we obtain for $\theta_1 \leq \theta
\leq \theta_0$:
\begin{equation*}
U(\theta)= \{(-(1/2)e^t\sin(2\theta)+u)+(e^t\cos^2\theta)i \; | \; u_0
\leq u \leq u_1, \; t_0 \leq t \leq t_1 \} \subset \mathbb{H}.
\end{equation*}
We have that $U(\theta)$ has the same hyperbolic area as:
\begin{equation*}
V(\theta) = \{u\sec^2\theta + e^ti \; | \; t_0 \leq t \leq t_1, \; u_0
\leq u \leq u_1 \}.
\end{equation*}
This is so since $V(\theta)$ is obtained from $U(\theta)$ by the
hyperbolic isometry:
\begin{equation*}
T_{\theta}(z)= (\sec^2\theta)z + \frac{1}{2}e^t(\sec^2\theta)
(\sin 2\theta); \quad z \in \mathbb{H}, \; \theta_1 \leq \theta \leq \theta_0.
\end{equation*}
We thus obtain:
\begin{proposition}
\begin{equation*}
\text{A}_h(U(\theta)) = \text{A}_h(V(\theta))=
\Bigg[\int_{t_0}^{t_1}\int_{u_0}^{u_1} \frac{dx \; dy}{y^2}
\Bigg]\sec^2\theta = (u_1 -u_0)(e^{-t_0} - e^{-t_1}) \sec^2\theta
\end{equation*}
Therefore:
\begin{eqnarray*}
m(U) &=& \frac{1}{2}(u_1 -u_0)(e^{-t_0} -
e^{-t_1})\int_{\theta_0}^{\theta_1}\sec^2 \theta \; d\theta \\
&=& \frac{1}{2}(u_1 -u_0)(e^{-t_0} -
e^{-t_1})\int_{\theta_0}^{\theta_1}d(\tan\theta)\\
&=& \frac{1}{2}(u_1 -u_0)(e^{-t_0} -
e^{-t_1})\int_{\theta_0}^{\theta_1} dv \\
&=& \frac{1}{2}(u_1 -u_0)(e^{-t_0} - e^{-t_1})(v_1 -v_0)
\end{eqnarray*}
i.e.
\begin{equation}\label{equation5}
m(U) = \frac{1}{2}(u_1 -u_0)(v_1 - v_0)(e^{-t_0} - e^{-t_1})
\end{equation}
\end{proposition}
Another way to compute $m(\widetilde{U})$ is the following. Let $I
\subset \mathbb{R}^3$ be the cube:
\begin{equation*}
I = \{(u,t,v) \in \mathbb{R}^3 |u_0 \leq u \leq u_1\quad v_0 \leq v \leq
v_1\quad t_0 \leq t \leq t_1 \}.
\end{equation*}
Let $\varphi \colon I \to \widetilde{U}$ be the parametrization of
$\widetilde{U}$ given by $\varphi(u,t,v) \coloneqq \text{A}(u,t,v)$, where
$\text{A}(u,t,v)$ is given by (\ref{equation3}). Let
$\partial_u = \frac{\partial}{\partial u}$, $\partial_t =
\frac{\partial}{\partial t}$,
$\partial_v=\frac{\partial}{\partial v}$ be the standard vector fields in
$I$. Then, $(\varphi_*(\partial_u)_{(u,t,v)},
\varphi_*(\partial_t)_{(u,t,v)}, \varphi_*(\partial_v)_{(u,t,v)})$
is a basis for the tangent space $T_{\text{A}(u,t,v)}\tilde{G}$. Comparing
this basis with the basis
$(X_{\text{A}(u,t,v)}, Y_{\text{A}(u,t,v)}, Z_{\text{A}(u,t,v)})$ given by the
basic vector fields at $\text{A}(u,t,v)$, we see that the change of bases
is given by the matrix:
\begin{equation*}
\begin{bmatrix}
1 & 0 & 0 \\
-v & 1 & 0 \\
-v^{2}e^{-t} & ve^{-t} & e^{-t}
\end{bmatrix}
\end{equation*}
Since the determinant of this matrix is $e^{-t}$ we have:
\begin{eqnarray}\label{equation6}
\frac{1}{2}m(\widetilde{U}) &=& \frac{1}{2}\int_{u_0}^{u_1}
\int_{t_0}^{t_1} \int_{v_0}^{v_1} e^{-t} du \; dt \; dv \\
                        &=& \frac{1}{2}(u_1 - u_0)(v_1 -v_0)(e^{-t_0} -
                        e^{-t_1}). \nonumber
\end{eqnarray}
Let us consider now the geodesic flow $g_t \colon \widetilde{G} \to
\widetilde{G}$. We have for $t \in \mathbb{R}$ and $p \in \widetilde{G}$:
\begin{equation*}
g^*_t(Y_p) = e^tY_{g_{t}(p)} \quad g^*_t(Z_p) = e^{-t}Z_{g_{t}(p)},
\end{equation*}
where $g^*_t$ is the differential. Then:
\begin{equation*}
\|g^*_t(Y_p)\| = e^t\|Y_p\| \quad \|g^*_t(Z_p)\|= e^{-t}\|Z_p\|.
\end{equation*}
Hence, $\{g_t\}$ is an Anosov flow leaving invariant the splitting:
\begin{equation*}
T\widetilde{G} = E^{+} \oplus E^{-} \oplus E
\end{equation*}
where in this Whitney sum $E^{+}$, $E^{-}$ and $E$ are the line bundles
 spanned by $Y$, $Z$ and $X$ respectively. The fact that $g$ is Anosov
 implies  that it is structurally stable and its periodic orbits and
 dense.  This accounts for its very rich dynamics.
\par The differential of the geodesic flow acts on the canonical
framing as follows:
\begin{equation*}
g^*_t(X_p,Y_p,Z_p) = (X_{g_{t}(p)}, e^{t}Y_{g_{t}(p)},
e^{-t}Z_{g_{t}(p)})
\end{equation*}
Therefore, the Jacobian of $g_t$ is identically equal to one and the
geodesic flow preserves $\Omega$. A similar calculation shows that
$\Omega$ is also preserved by $h^{+}$ and $h^{-}$.  \par The circle
group $\mathcal{K}$ is the one-paremeter subgroup corresponding to
$\begin{bmatrix} 0 & 1\\ -1 & 0
\end{bmatrix} \in \mathfrak{sl}(2, \mathbb{R})$. Therefore the vector field
$W = Y -Z$ induces a free action of the circle on $\widetilde{G}$. The
foliations $\mathcal{F}^+$ and $\mathcal{F}^-$ tangent to $E^+ \oplus E$ and $E^-
\oplus E$ are respectively the unstable and stable foliations of the
geodesic flow and are also obtained by left translations of the two
copies of the affine group in $\widetilde{G}$.  Both foliations are
transverse to $W$ and every leaf of these foliations intersects each
circular orbit of $W$ in exactly one point. \emph{This is the
  geometric interpretation of the two Iwasawa decompositions}.  Given
any measure (or more generally, any Schwartz distribution) one can
disintegrate the given measure with respect to each of the foliations.
If $\Gamma \subset \widetilde{G}$ is any discrete subgroup, then the
vector field $W$ descends to $\widetilde{G}/\Gamma = M(\Gamma)$ and
induces a periodic flow which gives $M(\Gamma)$ the structure of a
Seifert fibration over a hyperbolic orbifold. The foliations $\mathcal{F}^+$
and $\mathcal{F}^-$ descend to $M(\Gamma)$ and their leaves are transverse
to the fibres of the Seifert fibration. If $m(M(\Gamma)) < \infty$
then every leaf is dense. With respect to the induced metric, each
leaf in any of the two foliations is isometric to $\mathbb{H}$. Formula
(\ref{equation5}) can be obtained directly by disintegration.\\

\noindent\textbf{Definition.} Let $x \in \widetilde{G}$. Let $a$, $b$
and $c$ be positive reals.  Then a \emph{standard box} or simply a
\emph{closed box}, denoted by $C(x; a,b,c)$, or simply by $C$ if the
parameters are understood, is the subset of $\widetilde{G}$ defined as
follows:
\begin{equation*}
C(x;a,b,c) = \bigg\{h^{-}_v(g_t(h^+_u(x))) \bigg| v \in [-\frac{a}{2},
  \frac{a}{2}], t \in[-\frac{b}{2}, \frac{b}{2}], u \in [-\frac{c}{2},
  \frac{c}{2}] \bigg\}.
\end{equation*}
We call $x$ the \emph{center of the box}.\\

\begin{remark}[]~
\begin{itemize}
\item [(i)] The image under the geodesic flow of a closed box is
  another closed box:
\begin{equation*}
g_t(C(x;a,b,c))= C(g_t(x); e^{-t}a,b,e^tc); \quad t\in \mathbb{R}.
\end{equation*}
\item[(ii)] Any left-translation of a box is also a box:
\begin{equation*}
\alpha C(x; a, b, c) = C(\alpha x; a,b,c); \quad \alpha \in \widetilde{G}.
\end{equation*}
\item[(iii)] Since the center and the parameters of the box can be
  chosen arbitrarily, it follows that the interiors of the boxes (i.e.
  the \emph{open boxes}) generate the topology of $\widetilde{G}$ and
  the $\sigma$-algebra of its Borel subsets.
\end{itemize}
\end{remark}
\par Let us compute the Haar measure of the box $C \coloneqq C(x;
a,b,c)$.  Using a left translation by $x^{-1}$ and the fact that
$\widetilde{G}$ is unimodular, it is enough to compute the volume with
parameters $(e; a,b,c)$ where $e$ is the identity element. But in this
case using formula (\ref{equation6}) we have:
\begin{lemma}
 $m(C(x;a,b,c))= ac(e^{b/2}- e^{-b/2}) = 2ac[\sinh(\frac{b}{2})]$.
\end{lemma}
\par Let us return now $\mathcal{F}^{+}$ and $\mathcal{F^{-}}$. If $x \in
\widetilde{G}$, let us denote by $L^+(x)$ and $L^-(x)$ the leaves of
$\mathcal{F}^+ $ and $\mathcal{F}^{-}$ which contain the point $x$. Explicitly:
\begin{eqnarray*}
L^+(x) &=& \{h_u^+(g_{t}(x))|u,t \in \mathbb{R}\} \\
L^-(x) &=& \{h_v^-(g_{t}(x))|v,t \in \mathbb{R}\}. \\
\end{eqnarray*}
\noindent\textbf{Definition.} Let $C = C(x;a,b,c)$ be a closed box.
Then the base of $C$ is:
\begin{equation*}
\beta (C) \coloneqq \Big\{ g_t(h_u^-(x)) |\; -\frac{c}{2} \leq u \leq
\frac{c}{2}, \; -\frac{b}{2}\leq t\leq \frac{b}{2} \Big\}
\end{equation*}
Clearly, we have $\beta(C) \subset L_{-}(x)$ and $A_h(\beta(C)) =
2c\sinh(\frac{b}{2})$.  \par From now on we will work in $G =
\text{PSL}(2, \mathbb{R})$ and in $M = \text{PSL}(2, \mathbb{R})/\text{PSL}(2,
\mathbb{Z})$.  In these two manifolds we have the unstable and stable
foliations $\mathcal{F}^+$ and $\mathcal{F}^-$.  \par We define standard boxes in
$M$ exactly the same way. In fact, they are the projections of the
boxes in $\widetilde{G}$. However, we will only considerer
\emph{embedded} boxes in $M$.  If $C = C(x;a,b,c) \subset G$ is a box,
then:
\begin{equation*}
m(C)= ac(\sinh(b/2)).
\end{equation*}
If $C= C(x;a,b,c) \subset M$, then its Haar measure is given by
\begin{equation*}
\overline{m} = \frac{3}{\pi^2} ac(\sinh(b/2)).
\end{equation*}
Let us recall the identification $\psi \colon \text{PSL}(2, \mathbb{R}) \to
T_1 \mathbb{H}$, which assigns to each M\"obius transformation $\sigma(z) =
\frac{az +b}{cz + d}$, the point in $T_1\mathbb{H}$ by the formula:
\begin{equation*}
\psi(\sigma) = \left(\sigma(i), -\frac{\sigma
'(i)i}{|\sigma'(i)|}\right).
\end{equation*}
The action of $\text{PSL}(2,\mathbb{R})$ on $T_1 \mathbb{H} = \{(z, \theta)| z \in
\mathbb{H}, \theta(\text{mod} 2\pi) \}$ is given by:
\begin{equation*}
\sigma(z, \theta) = (\sigma(z), \theta-2\text{arg}(cz + d)), \quad
\sigma(z) = \frac{az + b}{cz + d}.
\end{equation*}
(Recall that the angles are measured counter-clockwise from the
vertical.)  \par Let $L^+(e)$ and $L^{-}(e)$ be the unstable and
stable leaves through the identity $e \in G$.\\

\noindent\textbf{Definition.}  $\psi(L^+(e)) \coloneqq L^+ $ and
$\psi(L^-(e)) \coloneqq L^-$ are the \emph{basic unstable and basic
  stable leaves}, respectively:
\begin{eqnarray*}
L^+ &=& \{(x + iy, \pi) \in T_1\mathbb{H}\} \\
L^- &=& \{ (x + iy, 0) \in T_1\mathbb{H} \}.
\end{eqnarray*}\\

\noindent\textbf{Definition.}
An \emph{adapted box} $C \subset G$ is a standard box such that
$\beta(C) \subset L^-.$ \\

If $\alpha \in G $ and $C(\alpha; a,b,c)$ is an adapted box centered
at $\alpha$, then $\alpha = (x_0 + iy_0, 0)$ and
\begin{equation*}
\beta(C) = \{(x + iy, 0)|y_0e^{-\frac{b}{2}} \leq y \leq
y_0e^{\frac{b}{2}}, |x - x_0| \leq ay_0/2 \}
\end{equation*}
(See Figure \ref{Figure8}).

\begin{figure}[h]
\begin{center}
  \includegraphics[width=12cm]{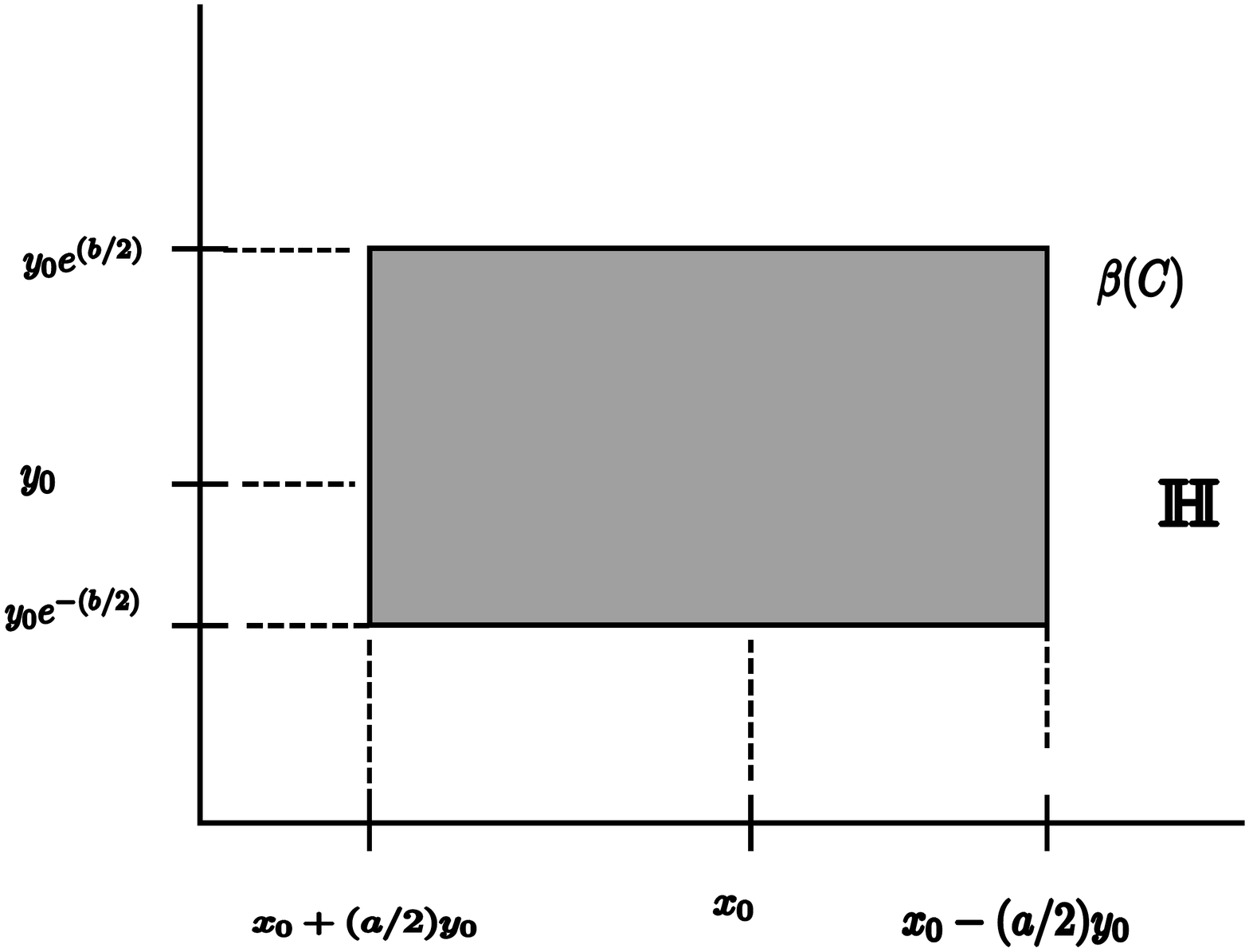}
\caption{}
\label{Figure8}
\end{center}
\end{figure}

\par We see that bases of adapted boxes can be identified with
rectangles in $\mathbb{H}$ whose sides are parallel to the coordinate axis.
We will not distinguish the base and the corresponding rectangle in
$\mathbb{H}$.  \par Let $p \colon G \to M $ be the covering projection and
let $\overline{P} \colon M \to S(\text{(PSL)}(2, \mathbb{Z}))$ be the Seifert
fibration onto the modular orbifold. Then $P(L^+)$ and $P(L^-)$ are
called the \emph{basic unstable and stable leaves } of the
corresponding geodesic flow.  $P(L^+)$ and $P(L^-)$ are the cylinders
mentioned before which contain all periodic orbits of $h^+$ and $h^-$
respectively.  \par Both $P(L^+)$ and $P(L^-)$ are dense in $M$. If $z
\in S(\text{PSL}(2,\mathbb{Z}))$ and $S^1(z) = \overline{P}^{-1}(z)$ denotes
the circular fibre over $z$, then
\begin{eqnarray*}
  L^+(\mathbb{Q}) &\coloneqq& S^1(z)\cap P(L^+)\\
  L^-(\mathbb{Q}) &\coloneqq&
  S^1 (z)\cap P(L^-)
\end{eqnarray*}
have the property that if $\alpha, \beta \in L^+(\mathbb{Q})$ (or
$L^-(\mathbb{Q})$) then there exists $\theta_0 \in \mathbb{Q}$ such that
\begin{equation*}
r_{\pi \theta_0}(\alpha) = \beta
\end{equation*}
where $r_{\theta} \colon M \to M $, $\theta \in \mathbb{R}$ is the periodic
flow induced by $W = Y - Z$.  \par This simple fact happens to be very
important for number theory.  The reason is clear (apart from the fact
that the rationals are involved): any invariant measure for $h^+$
corresponds to a Choquet measure on the image of the curve $D \colon
[0, \infty] \to C^*$. If this measure has compact support (or even if
the density of a Choquet measure at $D(0)$ and $D(\infty)$ decays very
rapidly as $y \to 0$ or $y \to \infty$) then the corresponding
probability invariant measure for $h^+$ is concentrated in $P(L^+)$.
\par Let $C$ denoted the set of all boxes in $M$ which are adapted,
i.e., $C \in \mathcal{C} \iff \beta(C) \subset P(L^-)$. Then $\mathcal{C}$ is a
basis for the topology of $M$. Hence, to see how the measure $m_y$,
approach the normalized Haar measure $\overline{m}$, it is enough to
estimate with precision $m_y(C)$ for all $C \in \mathcal{C}$.  \par For each
$t \in \mathbb{R}$, let $\Lambda_t$ be the horizontal line which is
parametrized by:
\begin{equation*}
\lambda_t(s) = s + e^{-t}i; \quad s \in \mathbb{R}.
\end{equation*}
Let $\hat{\lambda}_t \colon \mathbb{R} \to T_1\mathbb{H}$ be defined by
$\hat{\lambda}_t(s) = (\lambda_t(s),\pi)$. Thus $\lambda_t$
parametrizes the horocycle with equation $y= e^{-t}$ and $P \circ
\hat{\lambda}_t$ parametrizes with arcl-ength as parameter the unstable
horocycle orbit of period $y= e^{-t}$, i.e. it parametrizes:
\begin{equation*}
  \gamma_t = g_t(\gamma_0).
\end{equation*}
Let $A = \begin{bmatrix} a & b \\ c & d
\end{bmatrix} \in \text{SL}(2,\mathbb{Z})$, $c \neq 0 $.
Let $\bar{A}(z) = \frac{az + b}{cz + d}$ denote the corresponding
modular M\"obius transformation and let $A' \colon T_1 \mathbb{H} \to
T_1\mathbb{H}$ be the induced map in the unit tangent bundle. The image of
$\Lambda _t$ under $\bar{A}$ is the horocycle which is the circle
tangent to the real axis at the point $\frac{a}{c}$ and whose highest
point is:
\begin{equation}\label{equation7}
  z = \frac{a}{c} + e^t c^{-2}i; \quad (c\neq 0).
\end{equation}

This fact follows immediately since the point with biggest ordinate in
$\bar{A}(\Lambda_t)$ corresponds to the unique real number $s_0$ for
which $\frac{d}{ds}(\bar{A} \circ \lambda_t)\big|_{s=s_{0}}$ is real.
Hence $s_0 = -\frac{d}{c}$. We also obtain that $A'(\gamma_t)$
intersects $L^-$ only at the point:
\begin{equation*}
\left(\frac{a}{c} + e^t c^{-2}i,0\right) \in L^{-}, \quad (c \neq 0).
\end{equation*}
\par When $c = 0$, then $\bar{A}$ is a horizontal translation by an
integer and $\Lambda_t$ and $\gamma_t$ are kept invariant by $\bar{A}$
and $A'$ respectively.\\

\noindent\textbf{Definition.}
 For $t=0$ we have the basic horocycle:
\begin{equation*}
           H_1 = \{(x,y) \in \mathbb{H}| x \in \mathbb{R} \} = \Lambda_0.
    \end{equation*}
\par The \emph{basic horoball} is the boundary of $H_1$ in the extended
hyperbolic plane (the closure of $\mathbb{H}$ in the Riemann sphere):
\begin{equation*}
B_1 = \{(x,y) \in \mathbb{H}| \; \infty \geq y \geq 1 \}.
\end{equation*}
\par The images by elements of $\text{PSL}(2,\mathbb{Z})$ of the basic
horocycle are called the \emph{Ford circles } and the images of the
basic horoball are called the \emph{Ford discs}.  \par Two Ford discs
either coincide or else they are tangent at a point in $\mathbb{H}$ and
have disjoint interiors. The hyperbolic area of a Ford disc is one.
\par It follows from formula (\ref{equation7}) that the Ford discs are
tangent to the real axis at rational points (except for the basic
horoball which is tangent at the point at infinity), and every
rational point is a point of tangency. all these facts are important
for number theory (read the very last paragraph in Rademacher's
classic book in complex functions). (See Figure \ref{Figure9}).

\begin{figure}[h]
\begin{center}
  \includegraphics[width=8cm]{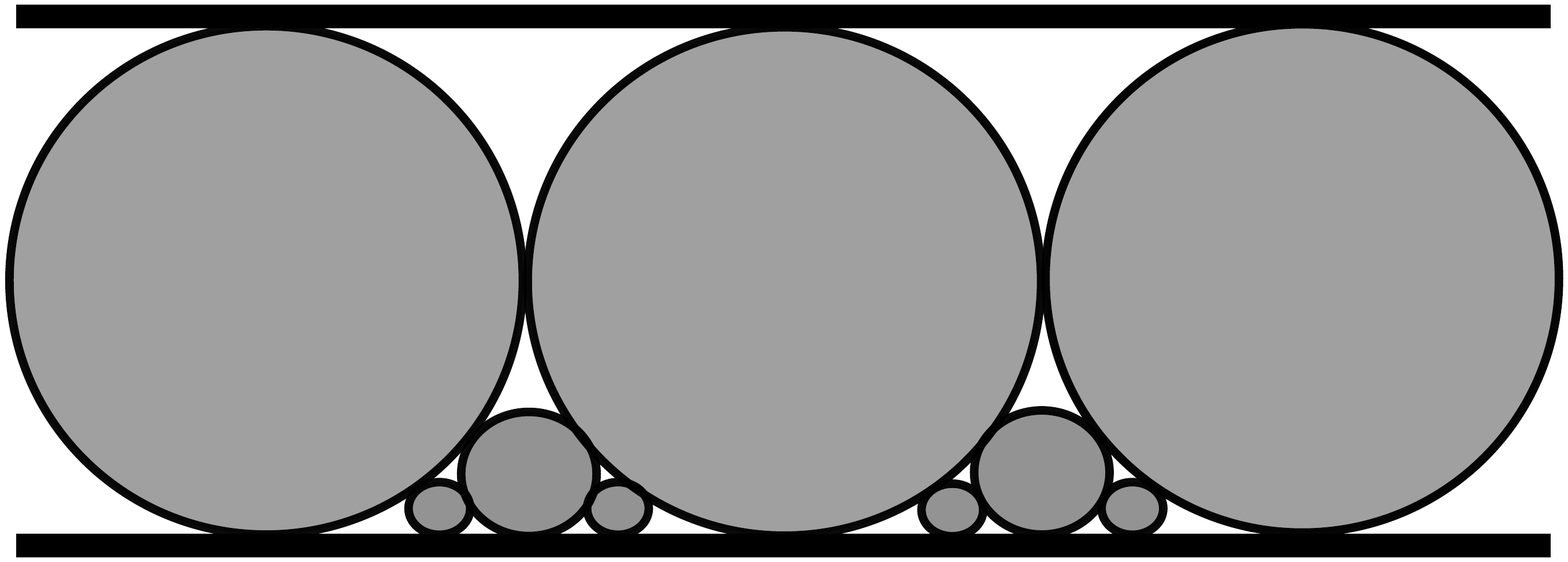}
\caption{}
\label{Figure9}
\end{center}
\end{figure}

\par Let $\mathcal{F}$ be the set of all Ford discs that intersect the strip
$0 \leq x \leq 1, \; y > 0$. For $r > 0$ let $F_r$ denote the subset
of Ford discs in $F$ that intersect the half-plane $y \geq
r^{-1}$. Then the circles in $F_r$ are tangent to the real axis at
exactly the rational points in $[0,1]$ which belong to the Farey
sequence of order $r^{1/2}$.  Therefore, its cardinality $|F_r|$ is
given by
\begin{equation*}
|F_r| = \sum_{n \leq r^{1/2}} \varphi(n) \coloneqq \Phi(r^{1/2}),
\end{equation*}
where $\varphi$ denotes Euler's \emph{totient function}. Everything
follows just by looking at formula (\ref{equation7}).
\par Let $Q$ be a rectangle which corresponds to the base of an adapted
box $C$. For each $t \in \mathbb{R}$ let $T_1 \colon \mathbb{H} \to \mathbb{H}$ be
defined by
\begin{equation}\label{equation8}
T_t(x,y) = (x, e^{-t}y)
\end{equation}
and let $Q_t = T_t(Q)$. Then $Q_t$ is the base of the adapted box
$g_{-t}(C)$.
\par As $t \to \infty $, $Q_t$ starts intersecting more and more Ford
discs. Let $n(t)$ denote the number of Ford discs whose highest point
is contained in $Q_t$, then $n(t)$ has the same growth type as that of
Farey  sequences contained in a fixed interval.
\par At this point it is clear the connection between Farey sequences,
the ergodic measures of the horocycle flow in $M$ and the Riemann
hypothesis. This is the link between Zagier's result \cite{Za} and the
following theorems of Franel \cite{Fr} and Landau \cite{Lan}:\\

\noindent\textbf{Teorema (Franel-Landau).}  \textit{Let $F_N = \{f_1,
  f_2, ..., f_{\Phi(N)} \}$ be the Farey sequence of order $N$
  consisting of reduced fractions $\frac{a}{b} \in (0,1]$, arranged in
order of magnitude.  Let $\delta_n = f_n - \frac{n}{\Phi(N)}$, $(n =
1, ..., \Phi(N))$ be the amount of discrepancy between $f_n$ and the
corresponding fraction obtained by equi-dividing the interval $[0,1]$
into $\Phi(N)$ equal parts.  Then, a necessary and sufficient
condition for the Riemann hypothesis is that:
\begin{equation*}
\sum_{i = 1}^{\Phi(N)} \delta_i^{2} = O\left(\frac{1}{N^{1-\epsilon}}
\right) \quad \text{ for all } \epsilon >0, \; (Franel).
\end{equation*}
 An alternative necessary and sufficient condition is that:}
\begin{equation*}
\sum_{i=1}^{\Phi(N)}|\delta_i| = O(N^{\frac{1}{2}+ \epsilon}) \quad
\text{ for all } \epsilon> 0, \; (Landau).
\end{equation*}\\

The subgroup of integer translations of the modular group identifies
$(x,y) \in \mathbb{H}$ with $(x + n, y) \in \mathbb{H}$, and $((x,y),0)$ with
$((x + n, y), 0)$; $n \in \mathbb{Z}$. Since we want adapted boxes in $M$
which are embedded, it is enough to consider standard boxes in $M$
which are projections of standard boxes in $T_1 \mathbb{H}$ whose bases lie
in the half-open strip $0 < x \leq 1$ (See Figure \ref{Figure10}).

\begin{figure}[h]
\begin{center}
  \includegraphics[width=8cm]{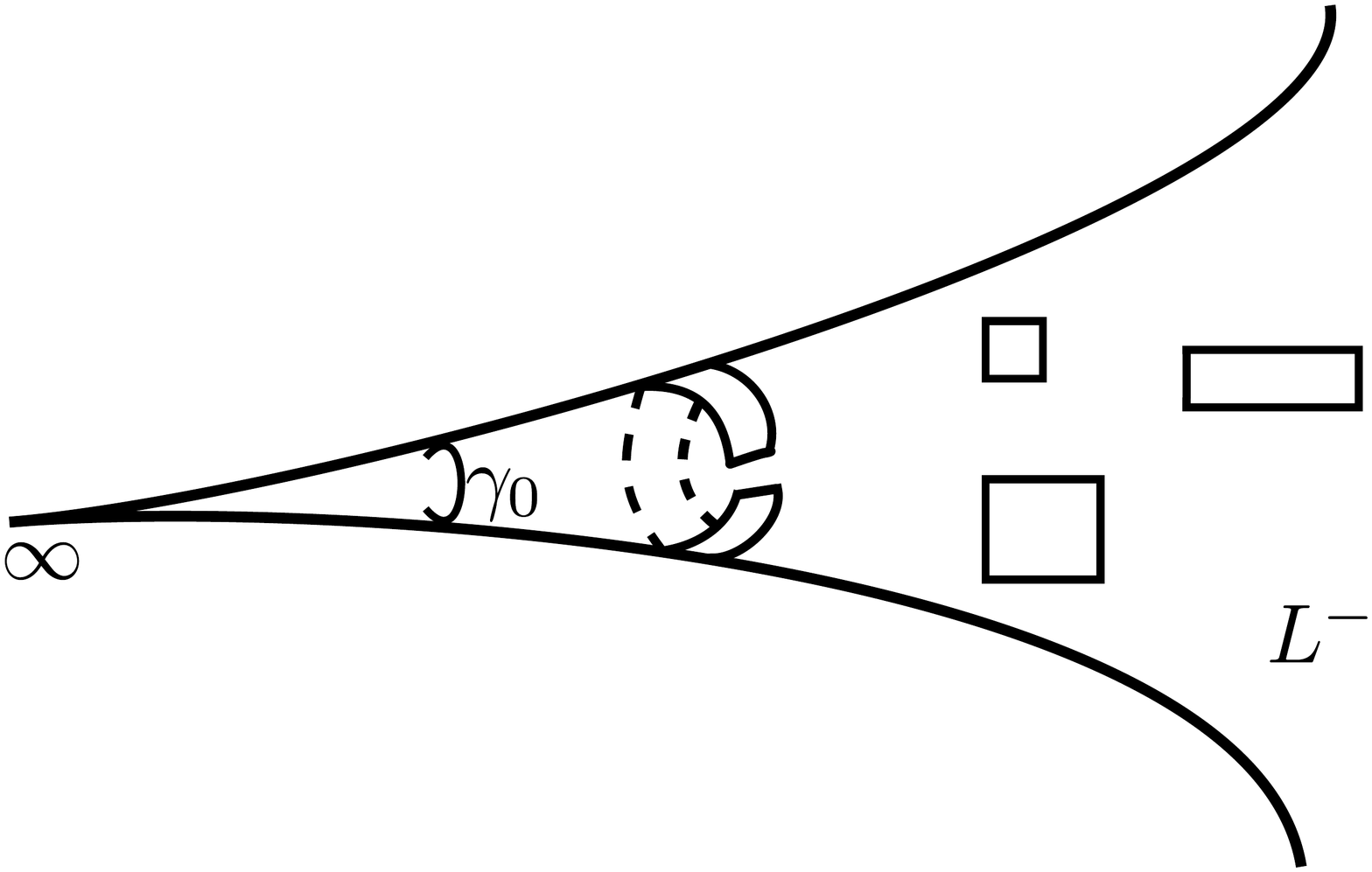}
\caption{}
\label{Figure10}
\end{center}
\end{figure}

\par Let $C \subset M$ be a box such that $\beta(C)$ is the rectangle
$Q \coloneqq Q(\alpha_1, \alpha_2; \beta_1, \beta_2)$, where $0 <
\alpha_1 < \alpha_2 \leq 1, 0< \beta_1 < \beta_2$, defined by:
\begin{equation*}
Q = \{(x,y) \in \mathbb{H}| \alpha_ \leq x \leq \alpha_2, \beta_1 \leq y
\leq \beta_2 \}.
\end{equation*}
\par we have that $\beta(g_{-t}(C)) = T_t(\beta(C))$ where $T_t$ is
given by formula (\ref{equation8}). Therefore:
\begin{eqnarray*}
n(t) &\coloneqq& \#\{g_t(\gamma_0) \cap \beta(C) \} = \#\{\gamma_0
\cap g_t(C)\} \\ &=& \# \{ \bar{A} \in \text{PSL}(2, \mathbb{Z}) \; \big| \;
A'(\gamma_t) \cap \{((x,y),0) \in T_1 \mathbb{H} \; \big| \; (x,y) \in Q \}
\neq \emptyset \} \\ &=& \# \bigg\{ A =
\begin{bmatrix}
	a & b \\
	c & d
\end{bmatrix}  \in \text{SL}(2, \mathbb{Z}), c \neq 0 \; \big| \; A'(\gamma_t)
\cap \{((x,y),0) \in T_1\mathbb{H} \; \big| \; (x,y) \in Q \} \neq
\emptyset \bigg\} \\ &=& \# \{ (a,c) \in \mathbb{Z}^+ \times \mathbb{Z}^+ \; |
\; c \neq 0, \{a,c \} = 1, (a/c + e^tc^{-2}i) \in Q \} \\ &=& \# \{
(a,c) \in \mathbb{Z}^+ \times \mathbb{Z}^+ \; | \; c \neq 0, \{ a,c\} = 1,
\alpha_1 \leq a/c \leq \alpha_2, e^{\frac{t}{2}}\beta_2^{-\frac{1}{2}}
\leq c \leq e^{\frac{t}{2}}\beta_1^{-\frac{1}{2}} \}.
\end{eqnarray*}
We thus have the following:\\

\noindent\textbf{Proposition.}  \textit{The number of points in which
  the horocycle orbit $g_t(\gamma_0)$ intersects the base of the box
  $\beta(C)$ as a function of $t$ is given by}
\begin{equation}\label{equation9}
n(t) = \# \{ (a,c) \in \mathbb{Z}^+ \times \mathbb{Z}^+ \; | \; c \neq 0, \{
a,c\} = 1, \alpha_1 \leq a/c \leq \alpha_2, e^{\frac{t}{2}}
\beta_2^{-\frac{1}{2}} \leq c \leq e^{\frac{1}{2}}
\beta_1^{-\frac{1}{2}} \}.
\end{equation}

\par Let $\mathbb{R}_+^2 = \{ (u,v) \in \mathbb{R} \; | \; v> 0 \}$ denote the
open upper half-plane with the Euclidean metric.  \par Let $\Psi
\colon \mathbb{R}_+^2 \to \mathbb{H}$, be the function:
 \begin{equation}\label{equation10}
 \Psi(u,v) = \frac{u}{v} + v^{-2}i, \quad (v > 0).
 \end{equation}
This simple function has the following remarkable properties:\\

\noindent\textbf{Proposition.}  \textit{$\Psi$ is an
  orientation-reversing diffeomorphism from the upper euclidean
  half-plane onto the hyperbolic plane. It sends rays emanating from
  the origin onto the family of geodesics which are vertical lines and
  the family of horizontal lines onto the family of horizontal
  horocycles.  The absolute value of the Jacobian of $\Psi$, with
  respect to the Euclidean and hyperbolic metric is 2. Therefore,
  $A_e(U) = \frac{1}{2}A_h(\Psi(U))$ where $U \subset \mathbb{R}_+^2$ is
  any open set and $A_e$, $A_h$ denote the Euclidean and hyperbolic
  areas, respectively. $\Psi$ sends the integer lattice points with
  relatively prime coordinates in the euclidean upper half-plane onto
  the points of tangency of the Ford circles in the hyperbolic
  plane.}\\

\par The usefulnes of this proposition is that \emph{it reduces the
  problem of counting the number of intervals in which a closed
  horocycle intersects a box to a euclidean lattice point counting}
\par Let $\Delta$ be any trapezium in $\mathbb{R}_+^2$ whose boundary
consists of two horizontal lines and two segments which are collinear
to the origin. Then $\Psi(\Delta)$ is a rectangle in $\mathbb{H}$ whose
sides are parallel to the coordinates axis(See Figure \ref{Figure11}).

\begin{figure}[h]
\begin{center}
  \includegraphics[width=10cm]{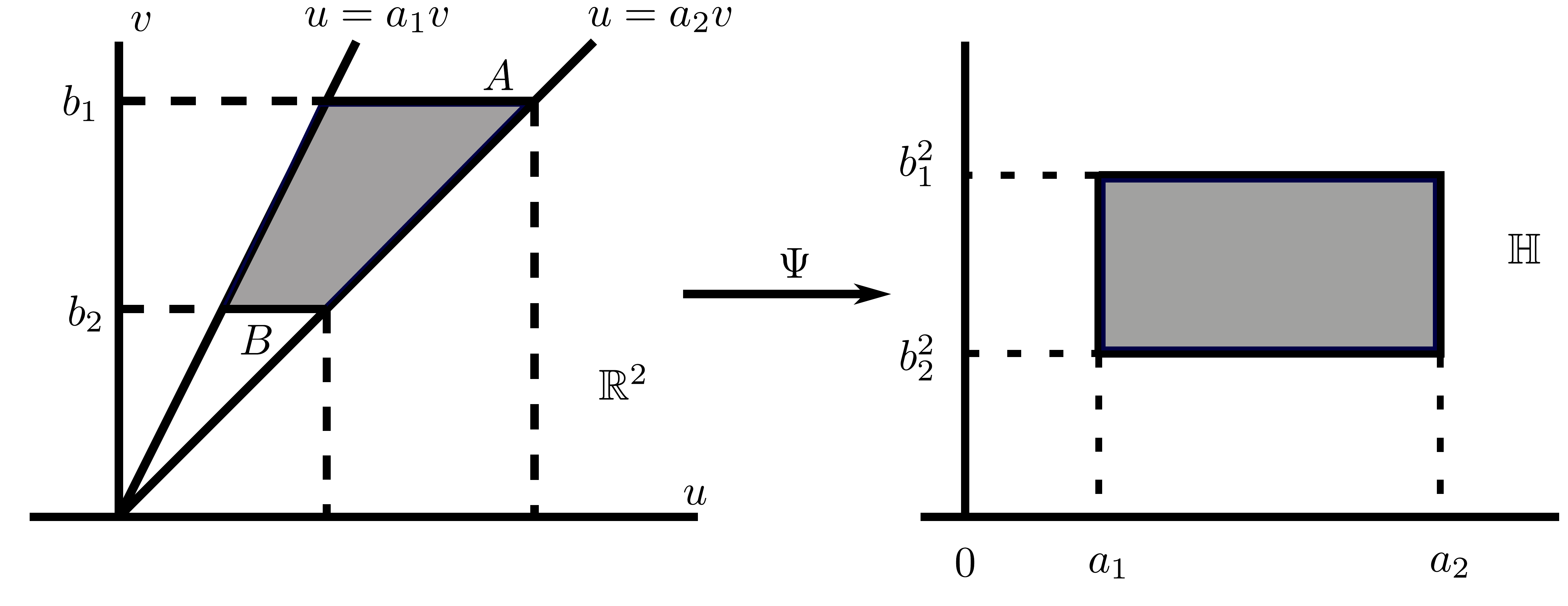}
\caption{}
\label{Figure11}
\end{center}
\end{figure}

\par If $\Delta$ is bounded by the lines $u=a_1v$, $u= a_2v$ and the
lines $v= b_1$, $v= b_2$, then $\Psi(\Delta) = \{(x,y) \in \mathbb{H}| a_1
\leq x \leq a_2, b_1^{-2} \leq y \leq b_2^{-2} \}$. then,
\begin{equation*}
A_e(\Delta) = \frac{(a_2 - a_1)(b_1^2 - b_2^2)}{2}
\end{equation*}
and
\begin{equation*}
-\frac{1}{2}A_h(\Psi(\Delta)) = \int_{b_1^{-2}}^{b_2^{-2}}
\int_{a_1}^{a_2} \frac{dx dy}{y^2} = \frac{a_2 - a_1}{2}
\int_{b_1^{-2}}^{b_2^{-2}} \frac{dy}{y^2} = \frac{1}{2}(a_2 -a_1)
(b_1^2 - b_2^2)
\end{equation*}
Therefore
\begin{equation}\label{equation11}
A_e(\Delta) = \frac{1}{2} A_h(\Psi(\Delta))
\end{equation}
Of course we knew (\ref{equation11}) since the Jacobian of $\Psi$ is
$-2$, but we wanted this fact explicit (in fact the formula
(\ref{equation11}) for all trapezia implies that the Jacobian is
$\pm2$).  \par For each $t \in \mathbb{R}$, let $\mu_{t} \colon
\mathbb{R}_{+}^{2} \to \mathbb{R}_{+}^{2}$ be the homothetic transformation:
\begin{equation}\label{equation12}
\mu_t(u,v) = (e^{t/2}u, e^{t/2}v).
\end{equation}
\par Let $\Delta(t) = \mu_t(\Delta(0))$, for $t \in \mathbb{R}$, where
$\Delta(0) = \Delta$ is the trapezium of the previous paragraph. Let
$T_t$ be the transformation defined by formula (\ref{equation8}), the
we have :
\begin{equation}\label{equation13}
\Psi \circ \mu_t = T_t \circ \Psi.
\end{equation}
\par Hence, by formula (\ref{equation9}) we have:
\begin{equation}\label{equation14}
n(t) = \# \{ (a,b) \in  \mathbb{Z}^+ \times \mathbb{Z}^+ |\; \{ a,c\} = 1, (a,b)
\in \Delta(t)  \}.
\end{equation}

\section{Main results}
\par Now, all that it is left is to estimate $n(t)$. We will need the
following:
\begin{theorem}[Mertens]\label{mertens}
Let $0 < \alpha_1 < \alpha_2 \leq 1$. For each $\ell > 2$, let
$\tau(\ell)$ be the triangle in $\mathbb{R}_+^2$ defined by
\begin{equation*}
\tau(\ell) = \{(u,v) \in \mathbb{R}_+^2 | v \leq \ell, \alpha_1 \leq
\frac{u}{v} \leq \alpha_2 \}.
\end{equation*}
Let $N(\ell)$ be the number of lattice points with relatively prime
integral coordinates contained in $\tau(\ell)$. The
\begin{equation*}
N(\ell) = \frac{6}{\pi^2} A_e(\tau(\ell)) + \eta(\ell)\ell\log \ell,
\end{equation*}
where $|\eta(\ell)|$ is bounded by $(2 + 2\sqrt{2}) (1 + 1/ \log2)<24$
for all $\ell \geq 2$ .
\end{theorem}
\begin{proof}
This result is classical and the method of proof starts with Gauss.
However, I will give a proof here since I will need the method for the
following lemmas. I will adapt the proof given in Chandrasekharan
[Ch. p. 59].  \par $A(\ell) = A(1) \ell^2 = (\alpha_2 -
\alpha_1)\ell^2/2$, and $p(\ell)= p(1)\ell$ be the Euclidean area and
perimeter of $\tau(\ell)$, respectively. Let
\begin{equation*}
K(\ell) = \{ (a,b) \in \mathbb{Z}^+ \times \mathbb{Z}^+ | (a,b) \in \tau(\ell)
\},
\end{equation*}
and let $\overline{N}(\ell) = |K(\ell)|$ be its cardinality. Then
\begin{equation*}
A(\ell) - \sqrt{2} p(\ell) \leq \overline{N}(\ell) \leq A(\ell) +
\sqrt{2}p(\ell), \text{ for all } \ell \geq 2.
\end{equation*}
Hence
\begin{equation}\label{equation15}
\overline{N}(\ell)= A(\ell) + \alpha(\ell) p(\ell)\quad \text{ where }
\; | \alpha(\ell)| \leq \sqrt{2} \text{ for all } \ell \geq 2.
\end{equation}
For $\ell \geq 2$:
\begin{equation*}
\overline{N}(\ell) = \sum_{(m,n) \in K(\ell)}1
\end{equation*}
(an empty sum will be by definition equal to zero).  \par Therefore
\begin{equation*}
\overline{N}(\ell) = \sum_{1 \leq d \leq \ell} \sum_{\substack{(m,n)
    \in K(\ell) \\ \{m,n\}=d}}1 \quad (\ell \geq 2).
\end{equation*}
\par Then, since $\{n,m \}= d \Leftrightarrow \{\frac{n}{d},
\frac{m}{d} \}=1$, it follows that there exists a bijective
correspondence between the sets
\begin{eqnarray*}
B_1(d) &=& \{(n,m) \in \tau(\ell) | \{(n,m)=d \} \text{ and }\\ B_2(d)
&=& \{(n',m') | (n'm')\in \tau(\ell/d), \{n',m' \}=1 \}
\end{eqnarray*}
where $(1 \leq d \leq \ell)$.
\par By definition, the number of
elements of $B_2(d)$ is equal to $N(\ell/d)$. Hence,
\begin{equation*}
\overline{N}(\ell) = \sum_{1\leq d \leq
\ell}N\left(\frac{\ell}{d}\right), \quad
(\ell \geq 2).
\end{equation*}
Applying the M\"obius inversion formula, we obtain
\begin{equation*}
N(\ell) = \sum_{1\leq d \leq \ell}\mu(d)\overline{N}
\left(\frac{\ell}{d}\right), \quad (\ell \geq 2).
\end{equation*}
where $\mu$ is the M\"obius function.  \par By (\ref{equation15}), we
have
\begin{equation*}
N(\ell) = \sum_{1\leq d \leq \ell} \mu (d) \bigg[ \frac{A(\ell)}{d^2}
  + \ell p(1) \alpha \bigg( \frac{\ell}{d} \bigg) \bigg(
  \frac{1}{d}\bigg) \bigg].
\end{equation*}
\par Since $\alpha(\cdot)$ is bounded for all $\ell >2$, and
$|\mu(d)|\leq 1$, we have
\begin{equation*}
-\sqrt{2} \sum_{1\leq d \leq \ell}\frac{1}{d} \leq \sum_{1 \leq d \leq
  \ell}\frac{\alpha(\frac{\ell}{d})\mu(d)}{d} \leq \sqrt{2} \sum_{1
  \leq d \leq \ell} \frac{1}{d}
\end{equation*}
By Euler: $\gamma \leq \sum_{1}^{\ell}\frac{1}{d} - \log{\ell} \leq
1$, where $\gamma = 0.5770...$, is Euler's constant. Therefore
\begin{equation*}
\sum_{1 \leq d \leq \ell}\frac{\alpha(\frac{\ell}{d})\mu(d)}{d} =
\beta(\ell)\log{\ell}
\end{equation*}
where $\beta$ is a function such that $|\beta(\ell)| \leq \sqrt{2}
\bigg( 1 + \frac{1}{\log{2}}\bigg) < 6$ for all $\ell \geq 2$.  \par
On the other hand,
\begin{equation*}
\sum_{1 \leq d \leq \ell}\frac{\mu(d)}{d^2} = (\zeta(2))^{-1} -
\sum_{d = [\ell]+1}^{\infty}\frac{\mu(d)}{d^2} = \frac{6}{\pi^2} -
\sum_{d = [\ell]+1}^{\infty}\frac{\mu(d)}{d^2}
\end{equation*}
and
\begin{equation*}
\bigg| \sum_{d = [\ell]+1}^{\infty}\frac{\mu(d)}{d^2}\bigg| <
\int_{[\ell]}^{\infty}\frac{du}{u^2} = \frac{1}{[\ell]}
\end{equation*}
(where $[\cdot]$ is the greatest-integer function).  \par Finally,
since $A(\ell)= A(1)\ell^2 < \frac{1}{2}\ell^2$, and $p(1) < 2+
\sqrt{2}$ we have that
\begin{equation*}
N(\ell) = \frac{6}{\pi^2}A_e(\tau(\ell) + \eta (\ell)\ell\log{\ell})
\end{equation*}
where
\begin{equation*}
|\eta(\ell)| \leq (2 + 2\sqrt{2})\bigg( 1 + \frac{1}{\log{2}} \bigg) <
24.
\end{equation*}
\end{proof}
\begin{corollary}
The following estimates hold:
\begin{equation*}
\lim\limits_{\ell \to \infty} \frac{N(\ell)}{A_e(\tau(\ell))} =
\frac{6}{\pi^2}
\end{equation*}
and
\begin{equation*}
\lim\limits_{\ell \to \infty}\frac{N(\ell)}{\ell^2} =
\frac{6A(1)}{\pi^2}
\end{equation*}
\end{corollary}

\noindent\textit{Remark.}  Let $ \breve{\tau}(\ell) \coloneqq \{(u,v)
\in \tau(\ell)| v < \ell \}$ denote $\tau(\ell)$ minus its base, and
let $\breve{N}(\ell)$ denote the number of lattice points with
relatively prime coordinates contained in $\breve{\tau}(\ell)$. Then
\begin{equation*}
\breve{N}(\ell) = N(\ell) + K(\ell)\ell, \quad where \; |K(\ell)| \leq
1.
\end{equation*}

\par Therefore
\begin{eqnarray}\label{equation16}
\breve{N}(\ell) &=& \frac{6}{\pi^2}A(\tau(\ell)) +
\breve{\eta}(\ell)\ell \log \ell, \quad \text{where}
\\ |\breve{\eta}(\ell)| & \leq & 2 + (2 + 2 \sqrt{2})(1 +
\frac{1}{\log 2}) \quad \text{ for all } \ell \geq 2 \nonumber
\end{eqnarray}
\par Let $\ell > 2$, $0 < \alpha_1 < \alpha_2 \leq 1$, $0 < \beta_1 <
\beta_2$. Let $\Delta(\ell)$ be the trapezium in $\mathbb{R}_+^2$ defined as
\begin{equation*}
\Delta(\ell) = \{(u,v) \in \mathbb{R}_+^2 | \alpha_1 \leq u/v \leq \alpha_2,
\beta_1\ell \leq v \leq \beta \}, i.e.,
\end{equation*}
\begin{equation*}
\Delta(\ell) = \tau(\beta_2l) - \breve{\tau}(\beta_1\ell).
\end{equation*}
\par Let $\widehat{N}(\ell)$ be the number of lattice points with
relatively prime coordinates contained in $\Delta(\ell)$. Then
\begin{equation*}
\widehat{N}(\ell) = N(\beta_2\ell) - \breve{N}(\beta_1 \ell).
\end{equation*}
Therefore, by Theorem (\ref{mertens}) and (\ref{equation16}) we have
\begin{corollary}\label{corollary3.5}
The following equality holds
\begin{equation*}
\widehat{N}(\ell) = \frac{6}{\pi^2}A_e(\Delta(\ell)) +
\hat{\eta}(\ell)\ell  \log \ell
\end{equation*}
where
\begin{equation*}
|\hat{\eta}(\ell)| \leq 40[\beta_2 - \beta_1 + \beta_2 \log \beta_2 -
  \beta_1 \log \beta_1 + 1]
\end{equation*}
for all $\ell \geq \text{max}\{2, 2\beta_1^{-1}, e^{\beta_{1}}\}$.
\end{corollary}
\begin{corollary}\label{corollary3.6}
The following formula holds:
\begin{equation*}
\lim\limits_{\ell \to \infty} \frac{\hat{N}(\ell)}{l^2} =
\frac{6}{\pi^2}A_e(\Delta(1)).
\end{equation*}
\end{corollary}
The following lemma depends only on two properties: the infinitude of
primes and the fact that $\varphi(p)= p-1$ if $p$ is a prime.
\begin{lemma}\label{lemma3.7}
Keeping the notation of Theorem (\ref{mertens}) we have:
\begin{equation*}
\overline{\lim\limits_{\ell \to \infty}}\bigg( \bigg|
\frac{N(\ell)}{A(\ell)} -
\frac{6}{\pi ^2} \bigg| \ell^{\frac{3}{2} -\epsilon}\bigg) = \infty,
\end{equation*}
for every $\epsilon$ such that $ -\infty < \epsilon < \frac{1}{2} $.
\end{lemma}
\begin{proof}
Suppose the lemma is not true for some $\epsilon \in (-\infty, 1/2)$.
Then there exists a bounded function $B_{\epsilon}(\ell)$, bounded for
all $\ell > 2$, such that
\begin{equation}\label{equation17}
\frac{N(\ell)}{A(\ell)} - \frac{6}{\pi^2} =
B_{\epsilon}(\ell)\ell^{-\frac{3}{2} + \epsilon}, \quad (\ell> 0).
\end{equation}
Define $H(\ell)$ by
\begin{equation*}
H(\ell) = \frac{N(\ell)}{A(\ell)},
\end{equation*}
then
\begin{equation}\label{equation18}
H(\ell + 1) = H(\ell) \frac{A(\ell)}{A(\ell + 1)} + \frac{\omega
(\ell)}{A(\ell + 1)} =
\bigg(\frac{\ell}{\ell + 1} \bigg)^2 H(\ell) +
\frac{\omega(\ell)}{A(1) (\ell + 1)^2},
\end{equation}
where $\omega(\ell)$ is the number of lattice points with relatively
prime coordinates contained in $\tau(\ell + 1) - \tau(\ell)$ i.e.
\begin{equation*}
\omega(\ell) = \{(a,b) \in \mathbb{Z}^+ \times \mathbb{Z}^+ | \{a,b \}=1, (a,b)
\in \tau (\ell + 1), b= [\ell + 1] \}.
\end{equation*}
If $[\ell + 1]$ is a prime, then $\omega(\ell) = \# \{a \in \mathbb{N}|\;
\alpha_1 \leq a /[\ell + 1] \leq \alpha_2, a < [\ell +1] \}$. \par
Now, $\tau(\ell + 1) - \tau(\ell)$ is a trapezium of unit height,
with its bottom side missing and whose non parallel sides have
positive slope greater than or equal to one (and whose parallel sides
are also parallel to the horizontal axis). Therefore it must contain
at least $\frac{(\alpha_2 - \alpha_1)(2\ell + 1)}{2} -2$ lattice
points and at most $\frac{(\alpha_2 - \alpha_1)(2\ell + 1)}{2} + 3$
such points.  \par Hence, if $[\ell + 1]$ is a prime , then
\begin{equation}\label{equation19}
\omega(\ell)= \frac{(\alpha_2 - \alpha_1)(2\ell + 1)}{2} +
\nu(\ell),
\quad \text{ where } |\nu(\ell)| \leq 3.
\end{equation}
From (\ref{equation17}) and (\ref{equation19}) we obtain
\begin{equation*} \bigg[\frac{6}{\pi^2} +
B_\epsilon(\ell)\ell^{-\frac{3}{2} + \epsilon} \bigg] \bigg[
\frac{\ell}{\ell + 1} \bigg]^2 + \frac{\omega(\ell)}{A(1)(\ell +
1)^2} = \frac{6}{\pi^2} + B_{\epsilon}(\ell +
1)(\ell+1)^{-\frac{3}{2} + \epsilon} \end{equation*}
Hence
\begin{eqnarray}\label{equation20}
L(\ell) &\coloneqq& \bigg[\frac{6}{\pi^2} \bigg[ 1- \bigg(
\frac{\ell}{\ell +
      1} \bigg)^2 \bigg] - \frac{\omega(\ell)}{A(1) (\ell + 1)^2}\bigg]
(\ell+1)^{3/2 - \epsilon} \\ &=& B_{\epsilon}(\ell)\bigg[ \frac{l}{
\ell +
    1}\bigg]^{\frac{1}{2} + \epsilon} - B_{\epsilon}(\ell + 1) \eqqcolon
R(\ell)\nonumber
\end{eqnarray}
\par Now considerer formula (\ref{equation20}) for all $\ell > 2$ such
that $[\ell + 1]$ is a prime . Then, using (\ref{equation19}) we have
\begin{equation*}
L(\ell) = \left[ \left(\frac{6}{\pi^2} -1 \right) \left( \frac{2\ell
+ 1}{(\ell
    + 1)^2} \right) - \frac{\nu(\ell)}{A(1) (\ell + 1)^2}\right]
(\ell+1)^{3/2 - \epsilon} \quad ([\ell + 1] \text{ a prime })
\end{equation*}
(recallig that $A(1) = (\alpha_2 - \alpha_1)/2$). But now we arrive to
a contradiction since under the hypothesis, $R(\ell)$ is bounded for all
$\ell > 2$ whereas $L(\ell)$ tends to infinity when $l$ tends to infinity by
a sequence $\{\ell_n\}$ such that $[\ell_n + 1] $ is prime.
\end{proof}
\begin{remark}
The fact that the set of all $\ell > 2$ for which $[\ell + 1]$ is a prime,
is infinite, is all that we used. Hence, it does not depend on the
triangle.
\end{remark}
\begin{remark} When $\epsilon$ above is very close to $1/2$, then
$\theta(\ell ) \coloneqq \left[ \frac{N(\ell)}{A(\ell)} - \frac{6}{\pi^2}
    \right] \ell^{3/2 - \epsilon}$ oscillates and becomes unbounded
  extremely slowly. Also it follows that $\theta_{\delta,
    c}(\ell)\coloneqq \left[\frac{N(\ell)}{A(\ell)} - c \right] \ell^{\delta +
    3/2}$ is $O(1)$ as $\ell \to \infty$ if and only if $c =
  \frac{6}{\pi^2}$ and $\delta < -\frac{1}{2}$. When $\alpha_1$ and
  $\alpha_2$ are rational and we let $\ell$ go to infinity through
  natural numbers then $N(\ell) / A(\ell)$ are (not very good) rational
  approximations of $\frac{6}{\pi^2}$.
\end{remark}
\par In the context  of Corollary \ref{corollary3.5}, let
$\stackrel{\circ}{\Delta}(\ell)$ denote the trapezium $\Delta(\ell)$
minus its open base:
\begin{equation*}
\stackrel{\circ}{\Delta}(\ell) = \{ (u,v) \in \mathbb{R}_+^2 | 0 < \alpha_1
\leq u/v \leq \alpha_2 \leq 1, \beta_1\ell < v \leq \beta_2\ell \} .
\end{equation*}
\par Then if $\stackrel{\bullet}{N}(\ell)$ denotes the number of lattice points with relatively
prime coordinates contained in $\stackrel{\circ}{\Delta}(\ell)$, we have
\begin{equation}\label{equation21}
\stackrel{\bullet}{N}(\ell) = \hat{N}(\ell) + \xi (l) \beta_1\ell, \text{ where
} -1 \leq \xi(\ell) \leq 0, \text{ for all } \ell > 2.
\end{equation}
\begin{corollary}\label{corollary3.15}
For every $\epsilon > 0$ such that $-\infty < \epsilon < 1/2$ we have
\begin{equation*}
\overline{\lim\limits_{\ell \to \infty}} \left( \left|
\frac{\hat{N}(\ell)}{A_e(\Delta (\ell))} - \frac{6}{\pi^2}
\right|\ell^{3/2 - \epsilon}  \right) = \infty.
\end{equation*}
\end{corollary}
\begin{proof}
Suppose the corollary is false for same $\epsilon$, $\epsilon \in
(-\infty, 1/2)$. Then, just as in Lemma (\ref{lemma3.7}), there exists
a function $B_{\epsilon}(\ell)$ bounded for all $\ell > 2$ such that
\begin{equation*}
\frac{\hat{N}(\ell)}{A_e(\Delta(\ell))} -\frac{6}{\pi^2} =
B_{\epsilon}(\ell)\ell^{-3/2 + \epsilon}.
\end{equation*}
Then by (\ref{equation21}), we have
\begin{equation*}
-\ell^{3/2 - \epsilon}\left|
\frac{\hat{N}(\ell)}{A_e(\Delta(\ell))}
-\frac{6}{\pi^2} \right| \leq
\ell^{3/2 + \epsilon} \left|
\frac{\stackrel{\bullet}{N}(\ell)}{A_e(\Delta(\ell))}
-\frac{6}{\pi^2}  \right| \leq \ell^{3/2 - \epsilon}\left|
\frac{\hat{N}(\ell)}{A_e(\Delta(\ell))} -\frac{6}{\pi^2} \right|.
\end{equation*}
Hence, we also have
\begin{equation*}
\overline{\lim\limits_{\ell \to \infty}}\ell^{3/2 + \epsilon} \left|
\frac{\stackrel{\bullet}{N}(\ell)}{A_e(\Delta(\ell))}
-\frac{6}{\pi^2} \right| < \infty \quad (\ell > 2, -\infty < \epsilon <
1/2)
\end{equation*}
\par Therefore, under the hypothesis, we conclude the existence of a
function $\stackrel{\bullet}{B}_{\epsilon}(\ell)$ whose absolute
value  is bounded by some
positive constant for all $\ell > 2$, and such that
\begin{equation*}
\stackrel{\bullet}{N}(\ell) = \frac{6}{\pi^2} A_{e}(\Delta(\ell)) +
A_{e}(\Delta(\ell))\stackrel{\bullet}{B}(\ell)\ell^{-3/2 + \epsilon}
\quad (\ell >
2, -\infty < \epsilon < 1/2)
\end{equation*}
\begin{equation*}
\stackrel{\bullet}{B}_{\epsilon}(\ell) = -\frac{6}{\pi^2} \ell^{3/2 -
  \epsilon}\quad \text{ if } \ell < \beta_2^{-1}.
\end{equation*}
\par Now, let $\delta = \beta_1 \beta_2^{-1}$. Then
\begin{eqnarray*}
N(\beta_2\ell) &=&
\sum_{n=0}^{\infty}\stackrel{\bullet}{N}(\delta^n\ell)
\\ &=& \sum_{n=0}^{\infty}\frac{6}{\pi^2} A_e(\Delta(\delta^n\ell)) +
A_e(\Delta(\delta^n\ell))\stackrel{\bullet}{B}(\delta^n \ell)(\delta^n
\ell)^{-3/2 + \epsilon} \\ &=& \frac{6}{\pi^2} A_e(\tau(\beta_2\ell)) +
\left( \sum_{n=0}^{\infty}A_e(\Delta(\delta^n\ell))
\stackrel{\bullet}{B}(\delta^n \ell) (\delta^n \ell)^{-3/2 +
  \epsilon}\right)\\
  & & +
  \sum_{n=m}^{\infty}\left(-\frac{6}{\pi^2}\right)A_e(\Delta(\delta^n\ell)) 
\end{eqnarray*}
where $m = \text{min}\{n \in \mathbb{N}|\; \ell < \delta^{-n}\beta_2^{-1} \}$ and
where $N(\cdot)$, $\tau(\cdot)$ are exactly as in Theorem
\ref{mertens}. Therefore, we have
\begin{eqnarray}\label{equation22}
N(\beta_2\ell) & = & \frac{6}{\pi^2}A_e(\tau(\beta_2\ell)) + \left\{
\sum_{n=0}^{\infty}A_e(\Delta(\delta^n
l))\stackrel{\bullet}{B}(\delta^n\ell)(\delta^n\ell)^{-3/2 +
\epsilon} \right\} \\
& & - \frac{6}{\pi^2} A_e(\tau(\delta^{m+1}(\beta_2\ell))).\nonumber
\end{eqnarray}
But (\ref{equation22}) implies that,
\begin{equation*}
N(\ell) = \frac{6}{\pi^2} A_e(\tau(\ell)) +
A_e(\tau(\ell))\bar{B}_{\epsilon}(\ell)\ell^{-3/2 + \epsilon}, \quad
\infty < \epsilon < 1/2,
\end{equation*}
where $\overline{B}_{\epsilon}(\ell)$ is bounded for all $\ell > 2$.
But this contradicts Lemma (\ref{lemma3.7}), therefore, Corollary
\ref{corollary3.6} must be true.
\end{proof}
\par Let us recall the functions $T_t$, $\Psi$ and $\mu_t$ given by
formulas (\ref{equation8}), (\ref{equation10}) and (\ref{equation12})
and connected by formula (\ref{equation13}): $\Psi \circ \mu_t = T_t
\circ \Psi$.  \par Let $C \subset M$ by any adapted box whose base
$\beta(C)$ corresponds to the rectangle
\begin{equation*}
Q \subset \mathbb{H}, \quad Q \coloneqq \{(x,y)| \alpha_1 \leq x \leq \alpha_2
\quad \beta_1 \leq y \leq \beta_2 \}
\end{equation*}
where $0 < \alpha_1 < \alpha_2 \leq 1$ and $0 < \beta_1 < \beta_2$. Let
$\delta(1)$ be the trapezium in $\mathbb{R}_{+}^2$ such that $\Psi(\delta(1))
= Q$. Then
\begin{equation*}
\Psi(\Delta(e^{t/2})) = Q_t = T_t(\Psi(\Delta(1))), \quad (t \in \mathbb{R})
\end{equation*}
where $Q_t$ is the rectangle which corresponds to the adapted box
$g_{-t}(C)$. As before, let $n(t) = \# \{\gamma_0 \cap
g_{-t}(\beta(C)) \}$. Then, (by (\ref{1.1}), (\ref{equation9}),
(\ref{equation14})), we have
\begin{equation*}
n(t) = \hat{N}(e^{t/2}).
\end{equation*}
By Corollary \ref{corollary3.5} we have that the function $n(t)$ must
be of the form:
\begin{equation}\label{equation23}
n(t) = \frac{6}{\pi^2}A_e(\Delta(e^{\frac{t}{2}})) +
\hat{\eta}(e^{\frac{t}{2}})\frac{te^{\frac{1}{2}}}{2},
\end{equation}
where $\hat{\eta}$ is a function bounded as follows:
\begin{equation*}
\big|\eta(e^{\frac{t}{2}})\big| \leq 40 (\beta_1^{-1/2} + \beta_2^{-1/2} +
\frac{1}{2}(\beta_2^{-1/2}\log \beta_2- \beta_1^{-1/2}\log\beta_1 + 1 ) )
\end{equation*}
\par for all $t \geq 2\text{max}\{\log2, \log2 + \frac{1}{2}\log\beta_2,
\beta_2^{-\frac{1}{2}} \}$.  \par From (\ref{equation23}) we obtain
\begin{equation*}
n(t) = \frac{3}{\pi^2} A_h(g_{-t}(\beta(C))) +
\hat{\eta}(e^{\frac{t}{2}})\frac{te^{\frac{t}{2}}}{2}.
\end{equation*}
And hence we obtain:
\begin{equation*}
n(t) = \frac{3}{\pi^2} A_h((\beta(C)))e^t +
\hat{\eta}(e^{\frac{t}{2}})\frac{te^{\frac{t}{2}}}{2}.
\end{equation*}
From Corollary (\ref{corollary3.15}) we obtain
\begin{equation*}
\overline{\lim\limits_{t \to \infty}}\left[ \left(
  \frac{n(t)e^{-t}}{A_h(\beta(C))} - \frac{3}{\pi^2} \right) e^{\alpha
    t} \right] = + \infty
\end{equation*}
for every $\alpha > 1/2$.  \par Changing the parameter: $y = e^{-t}$
and denoting as usual by $m_y$ the horocycle measure concentrated in
the unstable horocycle periodic orbit as period $1 / y \quad (y > 0
)$ we obtain using formula (1.1)
\begin{theorem}
\label{Theorem 3.18} Let $C \subset M$ be any
adapted box. Then
\begin{equation*} m_y(C) = \overline{m}(C) +
K_{C}(y)y^{1/2}\log y
\end{equation*} where $|K_C(y)|$ is bounded by
some positive constant for all $0 < y \leq 1/2$. Furthermore, this
positive constant can be chosen to be the same for every adapted box
contained in $C$.
\end{theorem}
(Compare with theorem \ref{T1.4})
\begin{corollary}
The following holds:
\begin{equation*}
\lim\limits_{y\to 0} m_y(C) = \overline{m}(C).
\end{equation*}
\end{corollary}
We have proved everything in Theorem (\ref{Theorem 3.18}). The fact
there exists a constant $k$ such that $|K_C(y)| \leq k$ for all $0< y
\leq 1/2$ and every box $C' \subset C$ follows from
(\ref{equation23}).
\begin{theorem}
Let $C$ be ay adapted box. Then
\begin{equation}\label{equation24} \overline{\lim\limits_{y \to
\infty}}[|m_y(C) - \overline{m}_y(C)|y^{\alpha}]= + \infty
\end{equation} for all $\alpha > 1/2$.
\end{theorem}
\begin{theorem}\label{Theorem 3.22}
Let $f \colon M \to \mathbb{R}$ be any continuous function with compact
support. Then, there exists a positive constant $K(f)$, depending only
of $f$, such that
\begin{equation}\label{3.23}
|m_y(f) - \overline{m}(f)| \leq K(f)y^{1/2}|\log y|, \quad (0 < y \leq 1/2).
\end{equation}
Furthermore, if $\alpha>1/2$, then
\begin{equation}\label{3.24}
\varlimsup_{y\to
0}[|m_{y}(\chi_{U})-\overline{m}(\chi_{U})|y^{-\alpha}]=\infty.
\end{equation}
\end{theorem}
\begin{proof}
(\ref{3.24}) is a direct consequence of (\ref{equation24}). For
(\ref{3.23}) one proceeds as follows: Let $\mathcal{B} = \text{supp}(f)$. Let
 $V_i = \{C_{1,i}, C_{2,i}, ... C_{n_i,i}
\}\ i = 1,2,...$ be a sequence of finite coverings of $\mathcal{B}$ by
adapted boxes. Let $F_i= \{g_1^i,...g_{n_i}^i \}$ be a smooth partition
of unity subordinated to $V_i(i = 1,2,...)$. For $i > 1$ suppose that
the maximum diameter of the boxes in $V_i$ is less than a Lebesgue
number for the covering $V_{i-1}$. If we consider $\{
g_1^i,...g_{n_i}^i \}$ we see that we can uniformly approximate $f$ by
a finite sum of functions which are constant on a box and zero outside
that box. The rest follows from Theorem (\ref{Theorem 3.18}).
\end{proof}

\section{APPENDIX A}
\textbf{The Ranking-Selberg Method and the Mellin transform of }
$\overline{m}_y$. 

In all that follows we will borrow from \cite{Za}
and \cite{Sa}.
\par Let $f \colon M \to \mathbb{R}$ be any $C^{\infty}$
function with compact support. Consider for $s \in \mathbb{C}$ with
$\mathfrak{R}(s)> 1$ the Mellin-type transform
\begin{equation*}
\langle E(s), f \rangle \coloneqq G_f(s) \coloneqq
\int_{0}^{\infty}m_y(f)y^{s-1} \frac{dy}{y}. 
\end{equation*}
\par We may think of $f$ as a function $f \colon T_1 \mathbb{H} \to \mathbb{R}$
which is $\Gamma$-invariant ($\Gamma = \text{PSL}(2, \mathbb{Z})$). Then 
\begin{equation*}
\langle m_y, f \rangle \coloneqq m_y(f) = \int_{0}^{1}f(x,y,0)dx
\end{equation*}
where $x$, $y$ and $\theta$ are the parameters of $\text{SL}(2,\mathbb{R})$
described before. 
\par Let $K_f = \text{sup}\{ \mathfrak{F}(s)| (z, \theta) \in
\text{supp}(f)  \}$. Then if $\mathfrak{R}(s) > 1$: 
\begin{equation}\label{equation26}
|G_f(s)| \leq \| f\|_{\infty} \frac{(K_f)^{\sigma - 1}}{\sigma - 1} \;
\sigma = \mathfrak{R}(s). 
\end{equation}
Therefore, we see that
\begin{itemize}
\item [(i)] The integral defining $G_f(\cdot)$ converges absolutely
  and uniformly in $\mathfrak{R}(s) > 1 + \epsilon$. Therefore $G_f(\;)$
  is holomorphic in $\mathfrak{R}(s) > 1$.
\item[(ii)] Because of inequality (\ref{equation26}) it follows that
\begin{equation*}
E(\cdot) \colon \{ \mathfrak{R}(s) > 1 \} \rightarrow [C_c^{\infty}]^*
= \mathcal{D}(M)
\end{equation*}
defines a weakly holomorphic function with values in the distribution
space, $\mathcal{D}(M)$, of $M$. In fact, (4.1) implies that for all $s_0$,
with $\mathfrak{R}(s_0) > 1$, $E(s_0)$ is complex valued infinite
measure concentrated in $P(L^+)$ and invariant under $h_t^+$. Where
$h_t^+$ acts on distributions $\mathcal{S}$ as follows:
\begin{equation*} \langle h_t^+ \mathcal{S}, f \rangle = \langle \mathcal{S}, f
\circ h_{-t}^+ \rangle, \; f\in C_c^{\infty}(M), t \in \mathbb{R}.
\end{equation*}
\item[(iii)] For every $f\in C_c^{\infty}(M)$, $G_g(\cdot)$ can be
  extended as a meromorphic function to all of $\mathbb{C}$ with no
  singularities in $\mathfrak{R}(s) > 1/2$ except for a simple pole at
  $s=1$ with residue
\begin{equation*}
\text{Res}(G_f(s))\big|_{s=1} = \overline{m}(f).
\end{equation*}
\item[(iv)] The growth on vertical lines $t \mapsto \sigma + it$ is
  controlled by the growth of
\begin{equation*}
\phi(s)= \frac{\pi^{1/2} \Gamma (s - 1/2)\zeta(2s
-1)}{\Gamma(s)\zeta(2s)}. 
\end{equation*}
\item[(v)] $E(\cdot) \colon \mathbb{C} \to \mathcal{D}(M) \cup \infty$ has the
property that $E(s)$ is a distribution of finite order ($s$ not a
pole); $E(S)$ satisfies the following functional equation:  
\begin{equation*}
E(s) = \mathcal{H}(s, \cdot) * E(1 -s)
\end{equation*}
where * denotes convolution and
\begin{equation*}
\mathcal{H}(s, \theta) = \left[ \frac{\Gamma(s)}{\pi^{1/2}\Gamma(s -1/2)}
  (\sin\theta/2)^{2s - 2} \right]\phi(s) = (\sin\theta/2)^{2s - 2}
\frac{\zeta(2s -1)}{\zeta(2s)}
\end{equation*}
and $\mathcal{H}(s, \cdot)$ acts on $2\pi$-periodic vector functions by
convolution in the $\theta$ variable.
\item [(vi)] Let $\mathfrak{a} = \text{sup}\{\mathfrak{R}(\rho);
  \zeta(\rho)= 0 \}$. Then, by Mellin inversion formula, we have for
  every $f \in C_{c}^{\infty }(M)$:
\begin{equation*}
m_y(f) = \frac{1}{2\pi i}\int_{-\frac{1}{2} -i \infty }^{-\frac{1}{2}
  + i \infty}G_f(s + 1)y^{-s}ds
\end{equation*}
(the validity of this inversion is shown in \cite{Sa}).
\end{itemize}
Changing variables and classical growth estimates of $\zeta(s)$,
$\Gamma(s)$ (Titchmarsh formulae (14.25) (14.2.6) p.283 \cite{Ti} we
have the validity of the following inversion:
\begin{equation}\label{equation27}
m_y(f) = \frac{1}{4 \pi i} \int_{\alpha -i \infty}^{\alpha + \infty}
G_f(\frac{s}{2})y^{1 - \frac{s}{2}}, \quad \text{ For } \alpha >
\mathfrak{a}.
\end{equation}
\par The right-hand side of (\ref{equation27}) is a superposition of
functions which are $\text{O}(1 - \frac{\mathfrak{a}}{2} - \epsilon)$
for all $\epsilon > 0$ ($y \to 0$) hence, $m_y(f)$ has the same order
for all $f \in C_{c}^{\infty }(M)$. This is the connection with the
Riemann Hypothesis found by Zagier. To prove the Riemann Hypothesis it
is enough to find a $C^2$ function, $f$, with compact support in $M$
such that the error term $|m_y(f) - \overline{m}_y(f)|$ can be made
$o(y^{3/4 - \epsilon})$ for all $\epsilon > 0$ and whose Mellin
transform has no zeroes on the critical strip.\\

\textbf{Eisenstein Series}. Let $C^{\downarrow}(S(\Gamma))$ denote the
space of functions which decay very rapidly as the argument of the
function approaches the cusp. Namely: the space of functions $f \colon
\mathbb{H} \to \mathbb{C}$ which are $\Gamma$-invariant and such that $f(x + iy)
= O(y^{-N})$ for all $N$. Then, since $f(x + iy)$ is periodic of period
one in the $x$-variable we may develop it into a Fourier series     
\begin{equation*}
f(x + iy) = \sum_{n \in \mathbb{Z}}\hat{f}_n(y)e^{2 \pi i n x}, \quad (y >
0).
\end{equation*}
Let $C(f,y) = \int_{0}^{1}f(x + iy)dx$ denote the constant term of its
Fourier expansion:
\begin{equation*}
C(f,y) = \hat{f}_0(y).
\end{equation*}
\par Let $\mathcal{M}(f,s)$ be the Mellin transform of $C(f,y)$:
\begin{equation*}
\mathcal{M}(f,s) = \int_{0}^{\infty}C(f,y)y^{s-1}\frac{dy}{y}, \quad
(\mathfrak{R}(s)>1). 
\end{equation*}
\par Using the fact $f$ is $\Gamma$-invariant and the fact that for
$\gamma \in \Gamma$ we go from a fundamental domain to the standard
domain, we have:
\begin{eqnarray*}
 \mathcal{M}(f,s) &=& \int_{0}^{\infty}\int_{0}^{1}C(f,y)y^{s-1}\frac{dy}{y}
 = \int_{0}^{\infty}\int_{0}^{1} f(x + iy)y^{s}\frac{dx \; dy}{y^2}
 \\ &=& \int_{\mathbb{H}/\Gamma}f(x)E(z,s)dz =
 \int_{S(\Gamma)}f(x)E(z,s)dz \quad (dz = \frac{dx dy}{y^{2}}).
\end{eqnarray*}
Hence
\begin{equation*}
\mathcal{M}(f,s) = \int_{S(\Gamma)}f(x)E(z,s)dz.
\end{equation*}
With this formula we see that $\mathcal{M}(f,s)$ enjoys the same properties
of $E(z,s)$:
\begin{itemize}
\item [(i)] $\mathcal{M}(f,s)$ has a meromorphic continuation to all of
  $\mathbb{C}$. It has a simple pole at $s = 1$.
\item [(ii)] $\text{Res}_{s=1} \mathcal{M}(f,s) =
\frac{3}{\pi}\int_{S(\Gamma)}f(u)du$. 
\item[(iii)] $M^*(f,s) \coloneqq \pi^{-s}\Gamma(s)\zeta(2s)\mathcal{M}(f,s)$
  is regular in $\mathbb{C} -\{ 0,1 \}$ and it satisfies the functional
  equation
\begin{equation*}
\mathcal{M}^*(f,s) = \mathcal{M}^*(f,1-s).
\end{equation*}

\end{itemize}
$E(z,s)$ is an Eisenstein Series:
\begin{equation*}
E(z,s) = \sum_{\gamma \in \Gamma_{\infty}/\Gamma}\mathfrak{F}(\gamma
z)^s = \frac{1}{2} \sum_{\substack{c,d \in \mathbb{Z} \\ \{ c,d \}=1}}
\frac{y^s}{|cz + d|^{2s}}
\end{equation*}

\section{APPENDIX B}

\textbf{Discrete measures and the Riemann-Hypothesis}. We refer to \cite{Ve} for this appendix. Many of the facts of this paper can be reduced to the study of discrete measures in the multiplicative group of the positive reals. 
Let $\mathbb R^\bullet=\left\{ y\in\mathbb R\,:\, t>0 \right\}$. Let $f:\mathbb N\to\mathbb C$  be any arithmetic function. For $y\in\mathbb R^\bullet$ consider the distribution $\nu_y:\text{C}^\infty_0(\mathbb R^\bullet)\to \mathbb C$, where $\text{C}^\infty_0(\mathbb R^\bullet)$ is the space of smooth functions with compact support,
defined by the formula:
$$
\nu_{y,f}=\underset{n\in \mathbb N}\sum\,f(n)\delta_{ny},
$$

\noindent where $\delta_x$ is the Dirac measure at the point $x\in\mathbb R^\bullet$, \emph{i.e.,}
$$
\nu_{y,f}(g)=\underset{n\in \mathbb N}\sum\,f(n)g(ny),\quad\quad g\in{\text{C}^\infty_0(\mathbb R^\bullet)}
$$

\noindent The Mellin transform
of $\nu_{y,f}(g)$ as a function of $y\in \mathbb R^\bullet$
$$
M_g(s)=\int_0^\infty\,y^{s-1}\nu_{y,f}(g)dy
$$

\noindent can be used to study the behavior of  $\nu_{y,f}(f)$ as $y\to0$. This is of course related to arithmetic since
the Mellin transform of $\nu_y$ involves the Mellin transform of $f$
and the Dirichlet series
$$
\underset{n\in \mathbb N}\sum\,\frac{f(n)}{n^s}
$$

\noindent For instance in \cite{Ve} there is the following:
\begin{theorem} Let $f$ above be Euler totient function $\varphi(n)$ that counts the number of integers lesser or equal to $n$ which are relatively prime to the integer $n$. Then,
\begin{enumerate}
\item For all $g\in{\text{C}^\infty_0(\mathbb R^\bullet)}$ 
$$
y^2\nu_{y,\varphi}(g)=\int_o^\infty\,\frac{6}{\pi^2}ug(u)\,du+ o(y)\quad \text{as}\,\, y\to0
$$ 
\item The Riemann hypothesis is true if and only if for for all $g\in{\text{C}^\infty_0(\mathbb R^\bullet)}$ 
$$
y^2\nu_{y,\varphi}(g)=\int_o^\infty\,\frac{6}{\pi^2}ug(u)\,du+ o(y^{\frac32-\epsilon})\quad \text{as}\,\, y\to0
\quad \forall \epsilon>0
$$ 

\end{enumerate}

\end{theorem}

\noindent\textit{Remark.} The author has found an interesting
connection with Hurwitz zeta function by considering horocycle
measures concentrated on equally spaced closed horocycles approaching
the cusp. Also one could give a geometric proof of the theorem of
Franel-Landau.

\section*{Acknowledgements}
\par I would like to thank S. G. Dani, E.
Ghys, S. Lopez de Medrano, C.  McMullen, B. Randol, D. Sullivan.
Finally I would like to thank  Dalid Rito Rodr\'iguez,  Antonia S\'anchez Godinez y Jos\'e Juan Zacar\'ias
for having typed the present version.

\bibliographystyle{plain}

\end{document}